\documentclass[a4paper]{article}
\usepackage{amsmath}
\usepackage{amsthm}
\usepackage{amssymb}
\usepackage{graphicx}
\usepackage{authblk}
\usepackage{hyperref}
\usepackage{multirow}
\usepackage{cite}
\usepackage{comment}

\newtheorem{theorem}{Theorem}[section]

\newtheorem{lemma}[theorem]{Lemma}

\newtheorem{remark}{Remark}

\newenvironment{vf}{\left\{\begin{array}{rcl}}{\end{array}\right.}

\newcommand{\includegraph}[2][]{\ifnum\pdfoutput=0\includegraphics[#1]{#2.eps}\else\includegraphics[#1]{#2.pdf}\fi}

\author[1]{Peter De Maesschalck}
\author[2]{Renato Huzak}
\author[3]{Ansfried Janssens\footnote{Corresponding author, {\tt ansfried.janssens@uhasselt.be}}}
\author[4]{Goran Radunovi\'{c}}

\affil[1,2,3]{Hasselt University, Campus Diepenbeek, Agoralaan Gebouw D, 3590 Diepenbeek, Belgium}
\affil[4]{University of Zagreb, Faculty of Science, Horvatovac 102a, 10000 Zagreb, Croatia}

\title{Minkowski dimension and slow-fast polynomial Li\'{e}nard equations near infinity }

\date{}

\begin{document}
\maketitle

\begin{abstract}
In planar slow-fast systems, fractal analysis of (bounded) sequences in $\mathbb R$ has proved important for detection of the first non-zero Lyapunov quantity in singular Hopf bifurcations, determination of the maximum number of limit cycles produced by slow-fast cycles, defined in the finite plane, etc. One uses the notion of Minkowski dimension of sequences generated by slow relation function. Following a similar approach, together with Poincar\'{e}--Lyapunov compactification, in this paper we focus on a fractal analysis near
infinity of the slow-fast generalized Li\'{e}nard equations $\dot x=y-\sum_{k=0}^{n+1} B_kx^k,\ \dot y=-\epsilon\sum_{k=0}^{m}A_kx^k$. We extend the definition of the Minkowski
dimension to unbounded sequences. This helps us better understand the fractal nature of slow-fast cycles that are detected inside the slow-fast  Li\'{e}nard equations and contain a part at infinity.
\end{abstract}

\tableofcontents

\section{Introduction}\label{Section-Introduction}
In this paper we give a fractal classification near infinity of the slow-fast Li\'{e}nard equations
 \begin{equation}
\label{model-Lienard}
    \begin{vf}
        \dot{x} &=& y-\sum_{k=0}^{n+1} B_kx^k \\
        \dot{y} &=&-\epsilon\sum_{k=0}^{m}A_kx^k,
    \end{vf}
\end{equation}
where $\epsilon\ge 0$ is the singular perturbation parameter kept small, $m,n\in\mathbb N_1$, $A_m\ne 0$ and $B_{n+1}\ne 0$. Our fractal classification will be based on Poincar\'{e}--Lyapunov compactification \cite{DH} and the notion of Minkowski dimension \cite{Falconer,Goran} of monotone sequences, converging to infinity and generated by so-called slow relation function (i.e. entry-exit relation) \cite{Benoit,DM-entryexit}. 
\smallskip

After a rescaling $(x,y,t)=(aX,bY,cT)$, with $a$, $b$ and $c$ depending only on $A_m$ and $B_{n+1}$, we can bring \eqref{model-Lienard} into
\begin{equation}
\label{model-Lienard1}
    \begin{vf}
        \dot{x} &=& y-\left(x^{n+1}+\sum_{k=0}^{n} b_kx^k \right)\\
        \dot{y} &=&-\epsilon\left(Ax^m+\sum_{k=0}^{m-1}a_kx^k\right),
    \end{vf}
\end{equation}
where we denote $(X,Y,T)$ again by $(x,y,t)$, $A=1$ if $m$ is even and $m\ne 2n+1$, $A=\text{sign}(A_m)$ if $m$ is odd and $m\ne 2n+1$, and where $A\ne 0$ if $m=2n+1$. See also \cite{DH} where the phase portraits of \eqref{model-Lienard1} with $\epsilon=1$ have been studied near infinity. In our slow-fast setting we let $\epsilon$ go to zero (for more details see the rest of this section).
\smallskip

Let us first explain the basic idea of our fractal approach. In planar vector fields one often deals with monodromic limit periodic sets, accumulated by spiral trajectories: homoclinic loop, limit cycle, focus, etc. Such limit periodic sets can produce limit cycles after perturbation, and the maximum number of limit cycles (i.e. the cyclicity) is typically studied using Poincar\'{e} map and looking at its fixed points. In \cite{EZZ,MRZ,ZupZub} it has been observed that the cyclicity is closely related to the density of orbits of the Poincar\'{e} map that converge to a fixed point of the Poincar\'{e} map (the fixed point corresponds to a focus, limit cycle or a more complex limit periodic set). See Fig. \ref{fig-chirp1}(a). Here, an important fractal dimension of the orbits comes into play: the Minkowski dimension, often called the box dimension. The Minkowski dimension measures the density of orbits (roughly speaking, if the Minkowski dimension increases, then the density increases and more limit cycles can be born). The Minkowski dimension can often be computed by comparing the (estimated) length of
$\delta$-neighborhood of orbits, as $\delta\to 0$, with $\delta$, $\delta^{1/2}$, $\delta^{1/3},\dots$ (see \cite{EZZ}). For some other applications of the Minkowski dimension see \cite{BurFalFRa,BoxVesna,ISOCEN} and references therein.
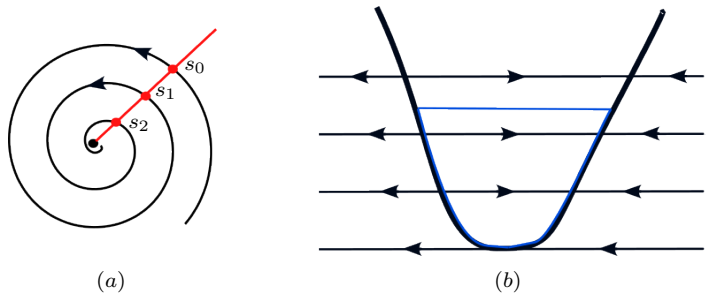
\begin{figure}[htb]
	\begin{center}
		\includegraphics[width=9.2cm,height=3.3cm]{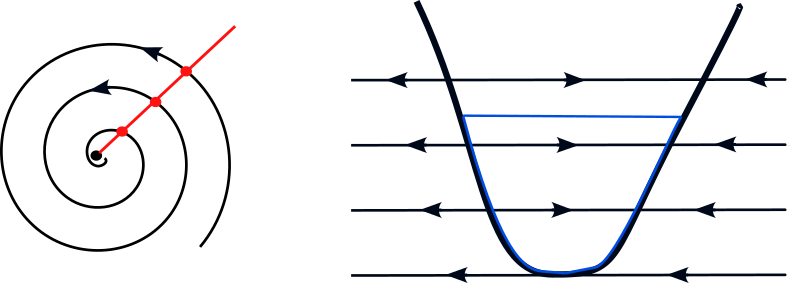}
		{\footnotesize
        \put(-80,-12){$(b)$}
        \put(-229,-12){$(a)$} \put(-196,70){$s_0$}
\put(-207,60){$s_1$}   \put(-217,48){$s_2$}     }
         \end{center}
	\caption{(a) The orbit $(s_l)_l$ of $s_0$ generated by the Poincar\'{e} map $P$ defined near a focus ($s_{l+1}=P(s_l)$). (b) A canard cycle (blue) consisting of a fast orbit, normally attracting and normally repelling portions of the curve of singularities and a contact point between them (see also Section \ref{sec-Motivation}).}
	\label{fig-chirp1}
\end{figure}

 In planar slow-fast setting we deal with degenerate limit periodic sets (containing a curve of singularities), defined at level $\epsilon=0$, where $\epsilon$ is the singular parameter (see \cite[Chapter 4]{DDR-book-SF}). Clearly, such limit periodic sets are non-monodromic and the Poincar\'{e} map is not defined for $\epsilon=0$ (Fig. \ref{fig-chirp1}(b)). Following \cite{BoxRenato,BoxDomagoj,BoxVlatko}, a natural candidate for a discrete one-dimensional dynamical system that generates orbits in the limit $\epsilon\to 0$ (instead of the Poincar\'{e} map) is the slow relation function/entry-exit relation computed along the curve of singularities (for precise definitions we refer to Section \ref{sec-Motivation}). In \cite{BoxRenato,BoxDomagoj,BoxVlatko} a fractal analysis of so-called canard limit periodic sets (simply called canard cycles) has been given. A canard cycle contains both attracting and repelling parts of the curve of singularities (see Fig. \ref{fig-chirp1}(b)). For more details about the definition of canard cycles, the reader is referred to \cite[Chapter 4]{DDR-book-SF}. From the Minkowski dimension of one orbit of the slow relation function, assigned to the canard cycle, we can read the number of limit cycles and type of bifurcations near the canard cycle (see \cite[Theorem 3]{BoxVlatko}). A fractal analysis of planar contact points (e.g. slow-fast Hopf point), based on the same slow relation function approach, can be found in \cite{BoxDarko,BoxNovo}. We point out that \cite{BoxNovo} gives a simple fractal method for detection of the first non-zero Lyapunov coefficient in slow-fast Hopf bifurcations. A fractal detection of the first non-zero Lyapunov coefficient in regular Hopf bifurcations has been studied in \cite{ZupZub}.
\smallskip

In the two cases mentioned above (canard cycles and contact points), we considered orbits, generated by the slow relation function, that converge monotonically to a canard cycle or a contact point (see also Section \ref{sec-Motivation}). Thus, the orbits in these two cases are bounded, and the main purpose of our paper is to study the Minkowski dimension of orbits, generated by the slow relation function inside the Li\'{e}nard family \eqref{model-Lienard1}, that monotonically go to infinity. Such orbits will be closely related to canard cycles inside \eqref{model-Lienard1}, at level $\epsilon=0$, containing a fast orbit located at infinity (see Fig. \ref{fig-caninf}). 
\smallskip

More precisely, we compute the Minkowski dimension of such unbounded orbits by using a quasi-homogeneous compactification  of \eqref{model-Lienard1} at infinity defined in \cite{DH}. We have the following three cases: $m<2n+1$, $m=2n+1$ and $m>2n+1$. When $m<2n+1$ or $m=2n+1$,  system \eqref{model-Lienard1} will be studied on the Poincar\'{e}--Lyapunov disc
of degree $(1, n+1)$. When $m>2n+1$, then we study \eqref{model-Lienard1} on the Poincar\'{e}--Lyapunov disc of degree $(2, m+1)$ (resp. of degree $(1, \frac{m+1}{2})$) for $m$ even (resp. $m$ odd). The reason for using these Poincar\'{e}--Lyapunov compactifications is two-fold. Firstly, we find the phase portraits of \eqref{model-Lienard1} near infinity and see that canard cycles (Fig. \ref{fig-caninf}), with a fast segment located at infinity, are possible only if $A>0$ and $m$ and $n$ are odd; and secondly, the Poincar\'{e}--Lyapunov compactifications enable us to (naturally) extend the definition of the Minkowski dimension to unbounded orbits attached to the canard cycles. We refer to \cite{GoranInf,Goran} for some other extensions (geometric inversion, tube function
at infinity, Lapidus zeta function at infinity, etc.) of the definition of the Minkowski dimension from
bounded sets in Euclidean spaces to unbounded sets. 
\smallskip

\begin{figure}[htb]
	\begin{center}
		\includegraphics[width=5.2cm,height=2.8cm]{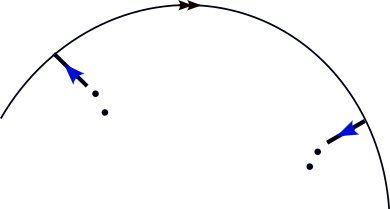}
         \end{center}
	\caption{In this paper we focus on the following behavior near infinity: a normally repelling portion of the curve of singularities with a slow dynamics (blue) directed towards the left corner, a fast orbit at infinity and a normally attracting portion with a slow dynamics (blue) directed away from the right corner. The corner points are normally hyperbolic when $m\le 2n+1$ and linearly zero when $m>2n+1$.}
	\label{fig-caninf}
\end{figure}

We point out that the shape of the curve of singularities of \eqref{model-Lienard1} can be very complex (i.e. canard cycles in Fig. \ref{fig-caninf} can be very diverse looking at their portion in the finite plane). In Section \ref{sec-Motivation} we state our main results for canard cycles that have the shape shown in Fig. \ref{fig-Motivation}. In Section \ref{sectionFAI} it will be clear that our fractal analysis near infinity can be applied to more general shapes, with a finite slow divergence integral in large compact sets in the phase space.   

We would like to stress that the phase portraits near infinity of \eqref{model-Lienard1} with $\epsilon=1$, studied in \cite{DH}, are qualitatively different from the phase portraits near infinity of \eqref{model-Lienard1} with $\epsilon\to 0$, presented in this paper. We therefore cannot use \cite{DH} to detect parts at infinity that may belong to the canard cycles for $\epsilon=0$. When $m>2n+1$, we need an additional blow-up to completely desingularize \eqref{model-Lienard1} near infinity, for $\epsilon\to 0$ (see Appendix \ref{appendix}). A special case with $(m,n)=(4,1)$ has been treated  in \cite{HuzakQuartic}.

\smallskip

Let $U\subset\mathbb R^N$ be a bounded set. Then we define the $\delta$-neighborhood of $U$:
$
U_\delta=\{x\in\mathbb R^N \ | \ d(x,U)\le\delta\}
$. We denote by $|U_\delta|$ the Lebesgue measure of $U_\delta$.
By the lower $s$-dimensional  Minkowski content of $U$, for $s\ge0$, we mean
$$
\mathcal M_*^s(U)=\liminf_{\delta\to 0}\frac{|U_\delta|}{\delta^{N-s}},
$$
and analogously for the upper $s$-dimensional Minkowski content $\mathcal M^{*s}(U)$ (we replace $\liminf_{\delta\to0}$ with $\limsup_{\delta\to0}$).
We define lower and upper Minkowski (or box) dimensions of $U$ as
$$
\underline\dim_BU=\inf\{s\ge0 \ | \ \mathcal M_*^s(U)=0\}, \ \overline\dim_BU=\inf\{s\ge0 \ | \ \mathcal M^{*s}(U)=0\}.
$$
If $\underline\dim_BU=\overline\dim_BU$, we call it the Minkowski dimension of $U$, and denote it by $\dim_BU$. For more details about the notion of Minkowski dimension  we refer the reader to  \cite{Falconer,tricot} and references therein. 
If $0<\mathcal M_*^d(U)\le\mathcal M^{*d}(U)<\infty$ for some $d$, then we say
 that $U$ is Minkowski nondegenerate. In this case we have $d=\dim_B U$. 
If $\Phi:U \subset \mathbb{R}^{N}\rightarrow \mathbb{R}^{N_1}$ is a bi-Lipschitz map (i.e., there exists a constant $\rho>0$ small enough such that
$\rho\left\|x-y\right\|\leq\left\|\Phi(x)-\Phi(y)\right\| \leq\frac{1}{\rho} \left\|x-y\right\|$,
for every $x,y\in U$), then $$\underline\dim_{B}U=\underline\dim_{B}\Phi(U), \ \overline\dim_{B}U=\overline\dim_{B}\Phi(U).$$
If $\Phi$ is a bi-Lipschitz map and $U$ is Minkowski nondegenerate, then $\Phi(U)$ is Minkowski nondegenerate (see \cite[Theorem 4.1]{ZuZuR^3}).
\smallskip

We describe a few basic notations  that we will use throughout this paper. When $(a_l)_{l\in\mathbb{N}}$ and $(b_l)_{l\in\mathbb{N}}$ are two sequences of positive real numbers converging to zero, we write $a_l\simeq b_l$, as $l\to\infty$,
if there exists a small positive constant $\rho$ such that $\frac{a_l}{b_l}\in[\rho,\frac{1}{\rho}]$ for all $l\in\mathbb{N}$. We use notation $I_-, J_-, \Phi_-,\dots$ (resp. $I_+, J_+, \Phi_+,\dots$) near attracting (resp. repelling) portions of curves of singularities.

\smallskip

The paper is organized as follows. In Section \ref{sec-Motivation} we motivate our fractal analysis of \eqref{model-Lienard1} and state the main results. In Section \ref{sec-statement}  we study the dynamics of \eqref{model-Lienard1} near infinity on Poincar\'{e}--Lyapunov disc (Section \ref{PLCdynamics}) and present the fractal analysis near infinity (Section \ref{sectionFAI}). We prove the main results in Section \ref{sec-proofs} using Section \ref{sectionFAI}.

\section{Motivation and statement of results}
\label{sec-Motivation}
In this section it is more convenient to write system \eqref{model-Lienard1} as
\begin{equation}
\label{model-Lienard2}
    \begin{vf}
        \dot{x} &=& y-F(x)\\
        \dot{y} &=&\epsilon G(x),
    \end{vf}
\end{equation}
where $F(x)=x^{n+1}+\sum_{k=0}^{n} b_kx^k$ and $G(x)=-Ax^m-\sum_{k=0}^{m-1}a_kx^k$. When $\epsilon=0$, system \eqref{model-Lienard2} has the curve of singularities $\mathcal C=\{y=F(x)\}$. Assume that 
\begin{equation}
\label{assum1}
F(0)=F'(0)=0, \qquad \frac{F'(x)}{x}>0, \ \forall x\in\mathbb{R}. 
\end{equation}
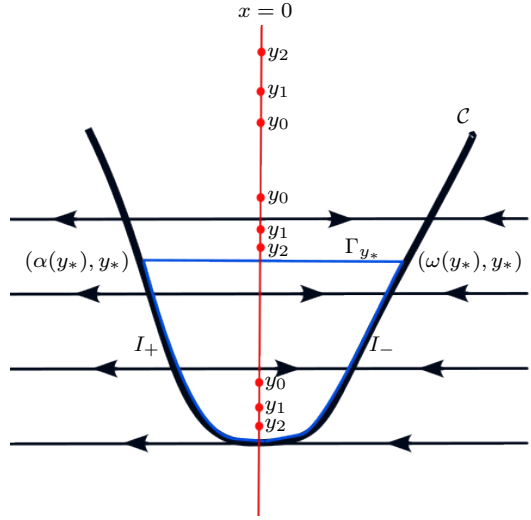
\begin{figure}[htb]
	\begin{center}
		\includegraphics[width=6.9cm,height=6.5cm]{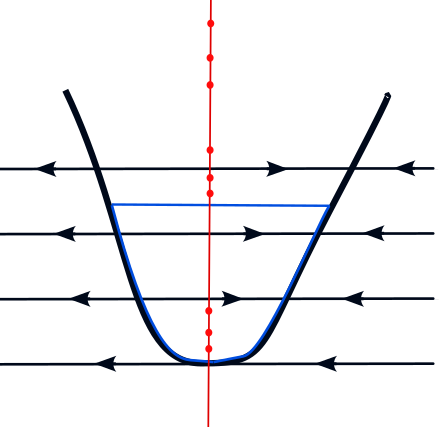}
		{\footnotesize
        \put(-29,148){$\mathcal C$}
        \put(-111,188){$x=0$}
        \put(-71,99){$\Gamma_{y_*}$}
        \put(-100,147){$y_0$}
        \put(-100,159){$y_1$}
        \put(-100,173){$y_2$}
\put(-100,119){$y_0$}
        \put(-100,106){$y_1$}
        \put(-100,99){$y_2$}
        \put(-101,49){$y_0$}
        \put(-101,40){$y_1$}
        \put(-101,33){$y_2$}
        \put(-62,62){$I_-$}
        \put(-150,62){$I_+$}
        \put(-44,94){$(\omega (y_*),y_*)$}
        \put(-191,94){$(\alpha (y_*),y_*)$}
             }
         \end{center}
 \caption{Orbits $U=\{y_0,y_1,\dots\}$ of $y_0$ by the slow relation function associated to \eqref{model-Lienard2} that monotonically converge to the contact point $\Gamma_{0}=\{(0,0)\}$ (see \cite{BoxNovo}), canard cycle $\Gamma_{y_*}$ in blue (see \cite{BoxRenato}) or $\Gamma_{\infty}$.}
	\label{fig-Motivation}
\end{figure}
The assumption in \eqref{assum1} implies that $n$ is odd, $b_0=b_1=0$ and that $\mathcal C$ has the general shape shown in Fig. \ref{fig-Motivation}. More precisely, the curve of singularities $\mathcal C$ consists of a normally attracting branch ($x>0$), a normally repelling branch ($x<0$), and a nilpotent contact point of Morse
type (i.e. of contact order $2$) at $(x,y)=(0,0)$. There are no other contacts with the horizontal fast foliation, and so-called slow dynamics of \eqref{model-Lienard2} along $\mathcal C$, $x\ne 0$, is therefore well defined:
\begin{equation}\label{slowdyn}
\frac{dx}{d\tau}=\frac{G(x)}{F'(x)}, 
\end{equation}
where $\tau=\epsilon t$ is the slow time ($t$ is the time in \eqref{model-Lienard2}). For more details about the notion of slow dynamics see \cite{DDR-book-SF,DHGener}. If we next assume that the slow dynamics \eqref{slowdyn} is regular for all $x\in\mathbb R$ and points from the attracting part of $\mathcal C$ to the repelling part of $\mathcal C$, i.e.
\begin{equation}
\label{assum2}
G(0)=0, \qquad \frac{G(x)}{x}<0, \ \forall x\in\mathbb{R}, 
\end{equation}
then it makes sense to study limit cycles of \eqref{model-Lienard2}, with $\epsilon>0$, Hausdorff close to canard cycle $\Gamma_{y_*}=\{(x,F(x)) \ | \ F(x)\le y_*\}\cup\{(x,y_*)\ | \ F(x)\le y_*\}$, defined for $y_*>0$ and $\epsilon=0$ (Fig. \ref{fig-Motivation}). From \eqref{assum2} it follows  that $m$ is odd, $a_0=0$ and $A>0$. Thus, $A=1$ if $m\ne 2n+1$ and $A>0$ if $m=2n+1$.
\smallskip

To study the number of limit cycles near $\Gamma_{y_*}$ we often use the slow divergence integral associated to $\Gamma_{y_*}$ (see e.g. \cite{DHGener}):
\[ I(y_*)=-\int_{\omega(y_*)}^{\alpha(y_*)}\frac{F'(x)^2}{G(x)}dx= I_-(y_*)-I_+(y_*)  \]
where 
\[I_-(y_*)=-\int_{\omega(y_*)}^{0}\frac{F'(x)^2}{G(x)}dx<0, \ I_+(y_*)=-\int_{\alpha(y_*)}^{0}\frac{F'(x)^2}{G(x)}dx<0,\]
$\omega(y_*)>0$, $F(\omega(y_*))=y_*$, $\alpha(y_*)<0$ and $F(\alpha(y_*))=y_*$ (Fig. \ref{fig-Motivation}). $I_-$ (resp. $I_+$) is the slow divergence
integral associated to the slow segment $[0,\omega(y_*)]$ (resp. $[\alpha(y_*),0]$). It is not difficult to see that $I_\pm'(y_*)<0$.
\smallskip

Following \cite{BoxRenato,BoxVlatko}, limit cycles near so-called balanced canard cycles can be studied using the Minkowski dimension. We say that $\Gamma_{y_*}$ is balanced if $I(y_*)=0$ (see \cite{DDR-book-SF}). Suppose that $\Gamma_{y_*}$ is balanced. Then the Implicit Function Theorem implies the existence of a function $S(y)$ such that $S(y_*)=y_*$ and 
\begin{equation}\label{slow-relation-def}
I_-(y)=I_+(S(y)),
\end{equation}
for all $y\sim y_*$ (we used the assumptions \eqref{assum1} and \eqref{assum2}). We call $S$ the slow relation function associated to the balanced canard cycle $\Gamma_{y_*}$. Let $y_0\sim y_*$ ($y_0\ne y_*$) be arbitrary and fixed and let $U$ be the orbit of $y_0$ by $S$, i.e. $U=\{y_l=S^l(y_0) \ | \ l\in\mathbb N\}$ where $S^l$ denotes $l$-fold composition of $S$. If we differentiate \eqref{slow-relation-def}, we get $S'>0$ (this implies that the orbit $U$ is monotone).  If $a_0=a_2=\dots=a_{m-1}=0$ and $b_1=b_3=\dots=b_n=0$ in \eqref{model-Lienard2}, then $I\equiv 0$, i.e. $S(y)=y$. Thus, in this case the orbit $U$ is a fixed point of $S$ and, therefore, cannot converge to $y_*$. 
\smallskip

Suppose now that $U$ (monotonically) converges to $y_*$ (when we break the above symmetry). Then $\dim_B U\in\{0,\frac{1}{2},\frac{2}{3}, \frac{3}{4},\dots\}$, and the balanced canard cycle $\Gamma_{y_*}$ can produce at most $l+1$ limit cycles if $\dim_B U=\frac{l-1}{l}$, with $l\in\mathbb N_1$ (see \cite[Theorem 3]{BoxVlatko}). This result gives a one-to-one correspondence between $\dim_B U$ and upper bounds for the number of limit cycles Hausdorff close to $\Gamma_{y_*}$. We point out that there exist simple formulas for numerical computation of $\dim_B U$ (see \cite[Section 7]{BoxVlatko}).
\smallskip

In \cite{BoxNovo} the fractal analysis of planar contact points has been given, under very general conditions. The assumptions \eqref{assum1} and \eqref{assum2} imply that the Li\'{e}nard system \eqref{model-Lienard2} has a slow-fast Hopf point at the origin $(x,y)=(0,0)$, the fractal analysis of which is covered by \cite{BoxNovo}. The function $S$ from \eqref{slow-relation-def} is well-defined for $y\sim 0$ and $y\ge 0$ ($I(0)=0,S(0)=0$) and, if we assume that the orbit $U$ of $y_0>0$ by $S$ converges to $0$ (Fig. \ref{fig-Motivation}), then $\dim_B U\in\{\frac{1}{3},\frac{3}{5}, \frac{5}{7},\dots\}$, and $\Gamma_{0}=(0,0)$ can produce at most $l$ limit cycles if $\dim_B U=\frac{2l-1}{2l+1}$, with $l\in\mathbb N_1$. For a precise statement of this result see \cite{BoxNovo}.
\smallskip

As discussed above, the idea was to generate a sequence, using $I_-(y_{l})=I_+(y_{l+1})$ (or $I_-(y_{l+1})=I_+(y_{l})$), which converges to the limit periodic set that we want to study. The following natural question arises: when do we have $y_l\to +\infty$ as $l\to +\infty$, and how do we define the Minkowski dimension of such an unbounded sequence? First, notice that 
\begin{equation}\label{asymptot}
\frac{F'(x)^2}{G(x)}=-\frac{(n+1)^2}{A}x^{2n-m}\left(1+o(1)\right), \ x\to \pm\infty.
\end{equation}
From \eqref{asymptot} and the definition of $I_\pm$ it follows that $I_\pm(y)\to-\infty$ as $y\to +\infty$ if $m<2n+1$ or $m=2n+1$, and that $I_\pm(y)$ converge to negative real numbers (we denote them by $I_\pm(+\infty)$) as $y\to +\infty$ if $m>2n+1$. Thus, if $m\le 2n+1$, then the function $S$ from \eqref{slow-relation-def} is well-defined for all $y\ge 0$ because $I_\pm$ are decreasing functions, $I_\pm(0)=0$ and $I_\pm(y)\to-\infty$ as $y\to +\infty$. If $U=\{y_l \ | \ l\in\mathbb N\}$ is the orbit of $y_0>0$ by $S$ such that $y_l\to +\infty$ as $l\to +\infty$, then we define the lower Minkowski dimension of $U$ by
\begin{equation}\label{def-dim-1}
\underline\dim_BU=\underline\dim_B\left\{\frac{1}{y_l^\frac{1}{n+1}} \ | \ l\in\mathbb N\right\},
\end{equation}
and analogously for the upper Minkowski dimension $\overline\dim_BU$. If $\underline\dim_BU=\overline\dim_BU$, then we write $\dim_BU$ and call it the Minkowski dimension of $U$.
\smallskip

If $m>2n+1$ and if we assume that
\begin{equation}\label{balanced-inf}
I_-(+\infty)-I_+(+\infty)=0,
\end{equation}
then $S$ is again well-defined for all $y\ge 0$ because $I_\pm$ are decreasing functions, $I_\pm(0)=0$ and $I_\pm(y)\to I_-(+\infty)$ as $y\to +\infty$. Similarly, if $U=\{y_l \ | \ l\in\mathbb N\}$ is the orbit of $y_0>0$ by $S$ with $y_l\to +\infty$ as $l\to +\infty$, then we define $\underline\dim_BU$ (and $\overline\dim_BU,\dim_BU$):
\begin{equation}\label{def-dim-2}
\underline\dim_BU=\underline\dim_B\left\{\frac{1}{y_l^\frac{2}{m+1}} \ | \ l\in\mathbb N\right\}.
\end{equation}
The definitions \eqref{def-dim-1} and \eqref{def-dim-2} do not depend on the choice of the initial point $y_0$ (see Theorems \ref{thm-1}--\ref{thm-3} stated below). Notice that, instead of computing the Minkowski dimension of the bounded sequence $\left(\frac{1}{y_l}\right)_{l\in\mathbb N}$, it is more natural to use the exponents of $y_l$ given in \eqref{def-dim-1} and \eqref{def-dim-2} which are related to degree of the Poincar\'{e}--Lyapunov disc. The sequences in \eqref{def-dim-1} and \eqref{def-dim-2} are obtained after the corresponding transformation in the positive $y$-direction (see Section \ref{sectionFAI}). This will simplify the computation of the Minkowski dimension in Theorems \ref{thm-1}--\ref{thm-3} (see Section \ref{sec-proofs}).

If one of the odd coefficients $b_{2j+1}$ in $F$ (resp. the even coefficients $a_{2j}$ in $G$) is nonzero, then we denote by $j_b$ (resp. $j_a$) the maximal $j$ with this property. If all the coefficients $b_{2j+1}$ (resp. $a_{2j}$) are zero, then $j_b$ (resp. $j_a$) is not defined and we write $b_o=0$ (resp. $a_e=0$). In the following theorems we assume that the Li\'{e}nard system \eqref{model-Lienard2} satisfies \eqref{assum1} and \eqref{assum2} and that the symmetry $a_e=b_o=0$ is broken (i.e., at least one of $j_a$ 
 and $j_b$ is well-defined). By ``for each $y_0$ sufficiently large" we mean for each $y_0>M$ where $M$ is a positive real number (large enough). 

Now we state the main result when $m<2n+1$.
\begin{theorem}[$m<2n+1$]\label{thm-1} Assume that $m<2n+1$ in \eqref{model-Lienard2}. The following statements are true.
\begin{enumerate}
    \item Suppose that $n-2j_b<m-2j_a$ or $a_e=0$. 
      \begin{itemize} 
      \item[(a)] If $m-n-2j_b-1< 0$ and $b_{2j_b+1}(m-n+2j_b+1)>0$ (resp. $b_{2j_b+1}(m-n+2j_b+1)<0$), then $I(y)\to -\infty$ (resp. $+\infty$), as $y\to+\infty$, and, for each $y_0$ sufficiently large, the orbit $U=\{y_l \ | \ l\in\mathbb N\}$ of $y_0$ by $S$ (resp. $S^{-1}$) tends (monotonically) to $+\infty$, $U$ is Minkowski nondegenerate and \begin{equation}\label{formm-1}\dim_B U=\frac{n-2j_b}{n+1-2j_b}.\end{equation} 
      \item[(b)] If $m-n-2j_b-1> 0$, then $I(y)$ converges to a real number $I_*$, as $y\to+\infty$. When $I_*<0$ (resp. $I_*>0$), for each $y_0$ sufficiently large the orbit $U$ of $y_0$ by $S$ (resp. $S^{-1}$) tends to $+\infty$, $U$ is Minkowski nondegenerate and  \begin{equation}\label{form-2}\dim_B U=\frac{2n+1-m}{2n+2-m}. \end{equation} Moreover, when $I_*=0$, for each $y_0$ sufficiently large the orbit $U$ of $y_0$ by $S$ (resp. $S^{-1}$) tends to $+\infty$ for $b_{2j_b+1}<0$ (resp. $b_{2j_b+1}>0$), $U$ is Minkowski nondegenerate and $\dim_B U$ is given in \eqref{formm-1}.
      \end{itemize}
    \item Suppose that $n-2j_b>m-2j_a$ or $b_o=0$. 
      \begin{itemize} 
      \item[(a)] If $2m-2n-2j_a-1<0$ and $a_{2j_a}<0$ (resp. $a_{2j_a}>0$), then $I(y)\to -\infty$ (resp. $+\infty$), as $y\to+\infty$, and, for each $y_0$ sufficiently large, the orbit $U$ of $y_0$ by $S$ (resp. $S^{-1}$) tends to $+\infty$, $U$ is Minkowski nondegenerate and \begin{equation}\label{form-3}\dim_B U=\frac{m-2j_a}{m+1-2j_a}.\end{equation} 
      \item[(b)] If $2m-2n-2j_a-1>0$, then $I(y)$ converges to a real number $I_*$, as $y\to+\infty$. When $I_*<0$ (resp. $I_*>0$), for each $y_0$ sufficiently large, the orbit $U$ of $y_0$ by $S$ (resp. $S^{-1}$) tends to $+\infty$, $U$ is Minkowski nondegenerate and $\dim_B U$ is given in \eqref{form-2}. If $I_*=0$, then for each $y_0$ sufficiently large the orbit $U$ of $y_0$ by $S$ (resp. $S^{-1}$) tends to $+\infty$ for $a_{2j_a}>0$ (resp. $a_{2j_a}<0$), $U$ is Minkowski nondegenerate and $\dim_B U$ is given in \eqref{form-3}.
      \end{itemize}
    \item Suppose that $n-2j_b=m-2j_a$ and $C:=b_{2j_b+1}(m-n+2j_b+1)-a_{2j_a}(n+1)\ne 0$. If $m-n-2j_b-1< 0$, then the same results as in Theorem \ref{thm-1}.1(a) are true if we replace $b_{2j_b+1}(m-n+2j_b+1)>0$ (resp. $b_{2j_b+1}(m-n+2j_b+1)<0$) from Theorem \ref{thm-1}.1(a) with $C>0$ (resp. $C<0$) and, if $m-n-2j_b-1> 0$, the same results as in Theorem \ref{thm-1}.1(b) are true if we replace $b_{2j_b+1}>0$ (resp. $b_{2j_b+1}<0$) from Theorem \ref{thm-1}.1(b) with $C>0$ (resp. $C<0$).
\end{enumerate}
\end{theorem}
We prove Theorem \ref{thm-1} in Section \ref{sectionproof1}.

\begin{remark}
    (a) When $m=1$ in Theorem \ref{thm-1}, then we deal with a classical Li\'{e}nard equation in \eqref{model-Lienard2}. Since $a_0=0$ and $b_1=0$, then we have $a_e=0$ and $n\ge 3$, and Theorem \ref{thm-1}.1(a) implies that the Minkowski dimension of $U$ can take only the following finite set of values: $\frac{n-2}{n-1},\frac{n-4}{n-3},\dots,\frac{3}{4},\frac{1}{2}$ (as $j_b$ increases from $1$ to $\frac{n-1}{2}$). We expect less limit cycles to be produced by $\Gamma_\infty$ as $j_b$ increases (i.e. the Minkowski dimension decreases) and, since there are $\frac{n-1}{2}$ different values for the Minkowski dimension, we conjecture that $\Gamma_\infty$ can produce at most $\frac{n-1}{2}$ limit cycles. Observe that the slow-fast Hopf point at $(x,y)=(0,0)$ in this classical Li\'{e}nard setting can have $\frac{n-1}{2}$ different values for the Minkowski dimension ($\frac{1}{3},\frac{3}{5},\dots,\frac{n-2}{n}$) and can produce at most $\frac{n-1}{2}$ limit cycles (see \cite{BoxNovo}).
    
    (b) Let's focus on Theorem \ref{thm-1} for fixed odd $m$ and $n$ with $m<2n+1$. It can be easily seen that in each statement of Theorem \ref{thm-1} $\dim_B U<\frac{2n+1-m}{2n+2-m}$ (resp. $\dim_B U=\frac{2n+1-m}{2n+2-m}$ and $\dim_B U>\frac{2n+1-m}{2n+2-m}$) indicates that the slow divergence integral $I(y)$ tends to $\pm\infty$ (resp. $I_*\ne 0$ and $I_*=0$) as $y\to +\infty$.
\end{remark}

The next theorem deals with the case where $m=2n+1$.
\begin{theorem}[$m=2n+1$]\label{thm-2}
Assume that $m=2n+1$ in \eqref{model-Lienard2}. Then $I(y)$ converges to a real number $I_*$, as $y\to+\infty$. Moreover, we have
\begin{enumerate}
    \item When $I_*<0$ (resp. $I_*>0$), for each $y_0$ sufficiently large, the orbit $U$ of $y_0$ by $S$ (resp. $S^{-1}$) tends to $+\infty$ and $\dim_B U=0$.
    \item Suppose that ($n-2j_b<2n+1-2j_a$ or $a_e=0$) and $I_*=0$. Then, for each $y_0$ sufficiently large, the orbit $U$ of $y_0$ by $S$ (resp. $S^{-1}$) tends to $+\infty$ for $b_{2j_b+1}<0$ (resp. $b_{2j_b+1}>0$), $U$ is Minkowski nondegenerate and $\dim_B U$ is given in \eqref{formm-1}.
    \item Suppose that ($n-2j_b>2n+1-2j_a$ or $b_o=0$) and $I_*=0$. Then, for each $y_0$ sufficiently large the orbit $U$ of $y_0$ by $S$ (resp. $S^{-1}$) tends to $+\infty$ for $a_{2j_a}>0$ (resp. $a_{2j_a}<0$), $U$ is Minkowski nondegenerate and $\dim_B U=\frac{2n+1-2j_a}{2n+2-2j_a}$.
    \item Suppose that $n-2j_b=2n+1-2j_a$, $C:=b_{2j_b+1}(n+2j_b+2)-a_{2j_a}\frac{n+1}{A}\ne 0$ and $I_*=0$. Then, for each $y_0$ sufficiently large, the orbit $U$ of $y_0$ by $S$ (resp. $S^{-1}$) tends to $+\infty$ for $C<0$ (resp. $C>0$), $U$ is Minkowski nondegenerate and $\dim_B U$ is given in \eqref{formm-1}.
   \end{enumerate}
\end{theorem}
We prove Theorem \ref{thm-2} in Section \ref{sectionproof2}. Note that in Theorem \ref{thm-2}.1 the Minkowski dimension of $U$ is zero. This means that the sequence in \eqref{def-dim-1} converges exponentially to
zero (see Section \ref{section-analysism=}). In the rest of Theorem \ref{thm-2} and in Theorem \ref{thm-1} and Theorem \ref{thm-3} the Minkowski dimension is always positive.

In the next theorem we assume that $m>2n+1$.
\begin{theorem}[$m>2n+1$]\label{thm-3}
Assume that $m>2n+1$ in \eqref{model-Lienard2} and that \eqref{balanced-inf} is true (i.e. $I(y)$ converges to $0$ as $y\to +\infty$). The following statements are true.
\begin{enumerate}
    \item Suppose that $n-2j_b<m-2j_a$ or $a_e=0$. Then, for each $y_0$ sufficiently large, the orbit $U$ of $y_0$ by $S$ (resp. $S^{-1}$) tends to $+\infty$ for $b_{2j_b+1}<0$ (resp. $b_{2j_b+1}>0$), $U$ is Minkowski nondegenerate and
    \begin{equation}\label{Th231}
\dim_B U=\frac{(n-2j_b)(m+1)}{(n-2j_b)(m+1)+2(n+1)}.
\end{equation}
    \item Suppose that $n-2j_b>m-2j_a$ or $b_o=0$. Then, for each $y_0$ sufficiently large, the orbit $U$ of $y_0$ by $S$ (resp. $S^{-1}$) tends to $+\infty$ for $a_{2j_a}>0$ (resp. $a_{2j_a}<0$), $U$ is Minkowski nondegenerate and
    \begin{equation}\label{Th232}
\dim_BU=\frac{(m-2j_a)(m+1)}{(m-2j_a)(m+1)+2(n+1)}.
\end{equation}
    \item Suppose that $n-2j_b=m-2j_a$ and $C:=b_{2j_b+1}(m-n+2j_b+1)-a_{2j_a}(n+1)\ne 0$. Then, for each $y_0$ sufficiently large, the orbit $U$ of $y_0$ by $S$ (resp. $S^{-1}$) tends to $+\infty$ for $C<0$ (resp. $C>0$), $U$ is Minkowski nondegenerate and $\dim_B U$ is given in \eqref{Th231} or \eqref{Th232}.
\end{enumerate}
\end{theorem}
 Theorem \ref{thm-3} will be proved in Section \ref{sectionproof3}.
 \begin{remark}
     If $n=1$ in Theorem \ref{thm-3} (Li\'{e}nard equations of degree $m$ with linear damping), then because $a_0=0$ and $b_o=0$, we have $\dim_BU=\frac{(m-2j_a)(m+1)}{(m-2j_a)(m+1)+4}$ with $j_a=1,\dots,\frac{m-1}{2}$. 
 \end{remark}
\smallskip

The case where $n-2j_b=m-2j_a$ and $C=0$ in Theorems \ref{thm-1}--\ref{thm-3} is a topic for further study. We strongly believe that similar methods can be used to deal with this case. 

For each statement in Theorems \ref{thm-1}--\ref{thm-3} it is possible to prove the existence of a Li\'{e}nard equation that satisfies assumptions of the statement. This will be presented in the doctoral thesis of Ansfried Janssens because of the length of this paper.

\section{Poincar\'{e}--Lyapunov compactification and fractal analysis near infinity}
\label{sec-statement}

\subsection{Poincar\'{e}--Lyapunov compactification and canard cycles near infinity}\label{PLCdynamics}

\subsubsection{The case $m<2n+1$}\label{m<-section-infty}
In this section we study the dynamics of \eqref{model-Lienard1} near infinity on the Poincar\'{e}--Lyapunov disc
of type $(1, n+1)$. In the positive $x$-direction we use the coordinate change 
\begin{equation}
x=\frac{1}{r}, \ y=\frac{\bar y}{r^{n+1}}\nonumber
\end{equation}
where $r>0$ is kept small and $\bar y$ is kept in a large compact set.  In the coordinates $(r,\bar y)$ system \eqref{model-Lienard1} becomes
\begin{equation}
\label{model-Lienardm<+}
    \begin{vf}
        \dot{r} &=& -r\left(\bar y-1-\sum_{k=0}^{n}b_kr^{n+1-k} \right)   \\
        \dot{\bar y} &=&-\epsilon r^{2n+1-m}\left( A+ \sum_{k=0}^{m-1}a_kr^{m-k}  \right)\\
        &&-(n+1)\bar y \left(\bar y-1-\sum_{k=0}^{n}b_kr^{n+1-k} \right),
    \end{vf}
\end{equation}
upon multiplication by $r^n$. For $\epsilon=0$, on the line $\{r=0\}$ system \eqref{model-Lienardm<+} has two singularities: $\bar y=0$ and $\bar y=1$. The eigenvalues of the linear part at $\bar y=0$ (resp. $\bar y=1$) are given by $(1,n+1)$ (resp. $(0,-(n+1))$). Thus, we have at $\bar y=0$ a hyperbolic and repelling node and at $\bar y=1$ a semi-hyperbolic singularity with the $\bar y$-axis as stable manifold and the curve of singularities $\bar y=1+\sum_{k=0}^{n}b_kr^{n+1-k}$ as center manifold.
\smallskip

Using asymptotic expansions in $\epsilon$ and the invariance under the flow we can see that center manifolds of \eqref{model-Lienardm<+}$+0\frac{\partial}{\partial \epsilon}$ at $(r,\bar y, \epsilon)=(0,1,0)$ are given by 
\[\bar y=1+\sum_{k=0}^{n}b_kr^{n+1-k}-r^{2n+1-m}\left(\frac{A}{n+1}+O(r)\right)\epsilon+O(\epsilon^2).\]
If we substitute this for $\bar y$ in the $r$-component of \eqref{model-Lienardm<+}, then we get the so-called slow dynamics 
\[r'=r^{2n+2-m}\left(\frac{A}{n+1} +O(r)\right),\]
upon desingularization (i.e. division by $\epsilon$ and $\epsilon\to 0$). Hence for $r>0$ and $r\sim 0$ the slow dynamics points away from the singularity $r=0$ if $A=1$ and $m$ odd or if $m$ is even. (Let's recall that $A=1$ if $m$ is even because $m<2n+1$.) The slow dynamics is directed towards $r=0$ if $A=-1$ and $m$ odd.

In the negative $x$-direction we have
\begin{equation}
x=\frac{-1}{r}, \ y=\frac{\bar y}{r^{n+1}}.\nonumber
\end{equation}
  System \eqref{model-Lienard1} changes (after multiplication by $r^n$) into
\begin{equation}
\label{model-Lienardm<-}
    \begin{vf}
        \dot{r} &=& r\left(\bar y-(-1)^{n+1}-\sum_{k=0}^{n}b_k(-1)^k r^{n+1-k} \right)   \\
        \dot{\bar y} &=&-\epsilon r^{2n+1-m}\left( A(-1)^m+ \sum_{k=0}^{m-1}a_k(-1)^k r^{m-k}  \right)\\
        &&+(n+1)\bar y \left(\bar y-(-1)^{n+1}-\sum_{k=0}^{n}b_k(-1)^k r^{n+1-k} \right) .
    \end{vf}
\end{equation} 
 When $\epsilon=r=0$, system \eqref{model-Lienardm<-} has two singularities, $\bar y=0$ and $\bar y=(-1)^{n+1}$. The eigenvalues of the linear part at $\bar y=0$ are given by $((-1)^{n},(-1)^{n}(n+1))$ and at $\bar y=(-1)^{n+1}$ by $(0,(-1)^{n+1}(n+1))$. Hence, if $n$ is odd (resp. even), we find at $\bar y=0$ a hyperbolic and attracting (resp. repelling) node and at $\bar y=(-1)^{n+1}$ a semi-hyperbolic singularity with the $\bar y$-axis as unstable (resp. stable) manifold. The curve of singularities is given by $\bar y=(-1)^{n+1}+\sum_{k=0}^{n}b_k(-1)^k r^{n+1-k}$. 
\smallskip

We obtain the slow dynamics along the curve of singularities near $(r,\bar y)=(0,(-1)^{n+1})$ in a similar fashion as in the positive $x$-direction:
\[r'=r^{2n+2-m}\left(\frac{(-1)^{m-n-1}A}{n+1} +O(r)\right).\]
Let $r>0$ and $r\sim 0$. Suppose that $n$ is odd. Then the slow dynamics points towards the origin $r=0$ if $A=1$ and $m$ odd, and away from $r=0$ if $A=-1$ and $m$ odd, or if $m$ is even. If $n$ is even, then the slow dynamics points to $r=0$ for $A=-1$ and $m$ odd, or for $m$ even, and away from $r=0$ for $A=1$ and $m$ odd.
\smallskip

There are no extra singularities in the positive (resp. negative) $y$-direction. After collecting all the information, we get the phase portraits near infinity of \eqref{model-Lienard1}, with $\epsilon=0$, including direction of the slow dynamics (see Fig. \ref{fig-m<}). 
\begin{figure}[htb]
	\begin{center}
		\includegraphics[width=10.4cm,height=7.7cm]{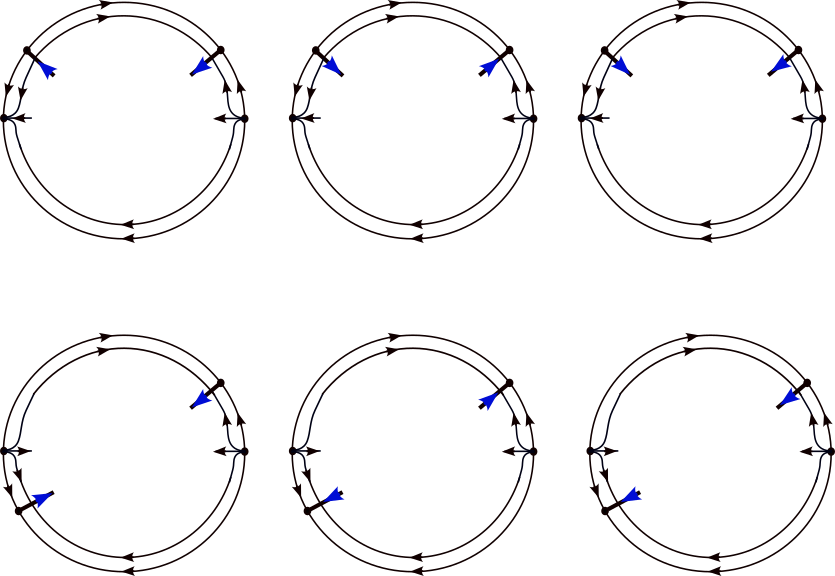}
		{\footnotesize
        \put(-85,113){$A=1$, $m$ even, $n$ odd}
        \put(-190,113){$A=-1$, $m$ odd, $n$ odd}
        \put(-292,113){$A=1$, $m$ odd, $n$ odd}
         \put(-85,-15){$A=1$, $m$ even, $n$ even}
        \put(-190,-15){$A=-1$, $m$ odd, $n$ even}
        \put(-292,-15){$A=1$, $m$ odd, $n$ even}
             }
         \end{center}
 \caption{The phase portraits near infinity of \eqref{model-Lienard1}, with $m<2n+1$ and $\epsilon=0$, and direction of the slow dynamics (blue).}
	\label{fig-m<}
\end{figure}
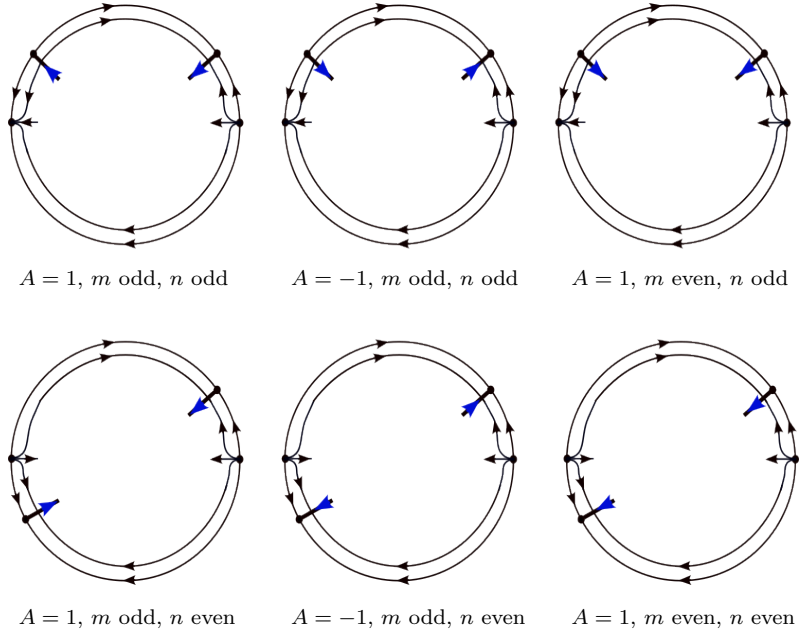
From Fig. \ref{fig-m<} it follows that, if \eqref{model-Lienard1}, with $m<2n+1$, has a canard cycle with a portion at infinity (Fig. \ref{fig-caninf}), at level $\epsilon=0$, then we have that $A=1$ and $m$ and $n$ are odd.
\smallskip

Section \ref{m<frac} is devoted to the fractal analysis of \eqref{model-Lienard1} near infinity, for $m<2n+1$, $A=1$ and $m$ and $n$ odd. It will be more convenient to present the fractal analysis in the positive $y$-direction where we use 
\begin{equation}
x=\frac{\bar x}{r}, \ y=\frac{1}{r^{n+1}},\nonumber
\end{equation}
bringing \eqref{model-Lienard1}, after multiplication by $r^n$, into 
\begin{equation}
\label{model-Lienardm<y+}
    \begin{vf}
        \dot{r} &=& \frac{\epsilon}{n+1}r^{2n+2-m}\left( A\bar x^m+ \sum_{k=0}^{m-1}a_kr^{m-k}\bar x^k  \right)  \\
        \dot{\bar x} &=& \frac{\epsilon}{n+1} r^{2n+1-m}\bar x \left( A\bar x^m+ \sum_{k=0}^{m-1}a_kr^{m-k}\bar x^k  \right)\\
        &&+1-\left(\bar x^{n+1}+\sum_{k=0}^{n}b_kr^{n+1-k}\bar x^k \right).
    \end{vf}
\end{equation}
For $n$ odd, both branches (the attracting and the repelling) of the curve of singularities are visible in the positive $y$-direction.

\subsubsection{The case $m=2n+1$}
We study the dynamics of \eqref{model-Lienard1} near infinity on the Poincar\'{e}--Lyapunov disc
of degree $(1, n+1)$. In the positive $x$-direction we use the coordinate change 
\begin{equation}
x=\frac{1}{r}, \ y=\frac{\bar y}{r^{n+1}}.\nonumber
\end{equation}
 In these new coordinates system \eqref{model-Lienard1} becomes
\begin{equation}
\label{model-Lienardm=+}
    \begin{vf}
        \dot{r} &=& -r\left(\bar y-1-\sum_{k=0}^{n}b_kr^{n+1-k} \right)   \\
        \dot{\bar y} &=&-\epsilon \left( A+ \sum_{k=0}^{2n}a_kr^{2n+1-k}  \right)
        -(n+1)\bar y \left(\bar y-1-\sum_{k=0}^{n}b_kr^{n+1-k} \right),
    \end{vf}
\end{equation}
after multiplication by a factor $r^n$. For $\epsilon=r=0$, system \eqref{model-Lienardm=+} has two singularities: $\bar y=0$, with the eigenvalues $(1,n+1)$ of the linear part, and $\bar y=1$ with the eigenvalues $(0,-(n+1))$ of the linear part. Hence we have at $\bar y=0$ a hyperbolic and repelling node and at $\bar y=1$ a semi-hyperbolic singularity with the $\bar y$-axis as stable manifold and the curve of singularities $\bar y=1+\sum_{k=0}^{n}b_kr^{n+1-k}$ as center manifold. Like in Section \ref{m<-section-infty}, we can compute the slow dynamics along the curve of singularities near $(r,\bar y)=(0,1)$:
\[r'=r\left(\frac{A}{n+1} +O(r)\right).\]
The slow dynamics is directed towards $r=0$ when $A<0$ and away from $r=0$ when $A>0$.

\smallskip

In the negative $x$-direction we use
\begin{equation}
x=\frac{-1}{r}, \ y=\frac{\bar y}{r^{n+1}}.\nonumber
\end{equation}
  System \eqref{model-Lienard1} changes (after multiplication by $r^n$) into
\begin{equation}
\label{model-Lienardm=-}
    \begin{vf}
        \dot{r} &=& r\left(\bar y-(-1)^{n+1}-\sum_{k=0}^{n}b_k(-1)^k r^{n+1-k} \right)   \\
        \dot{\bar y} &=&-\epsilon \left( -A+ \sum_{k=0}^{2n}a_k(-1)^k r^{2n+1-k}  \right)\\
        &&+(n+1)\bar y \left(\bar y-(-1)^{n+1}-\sum_{k=0}^{n}b_k(-1)^k r^{n+1-k} \right) .
    \end{vf}
\end{equation} 
When $\epsilon=r=0$, system \eqref{model-Lienardm=-} has two singularities, $\bar y=0$ and $\bar y=(-1)^{n+1}$, with the same eigenvalues as in the negative $x$-direction in Section \ref{m<-section-infty}. 
The slow dynamics along the curve of singularities $\bar y=(-1)^{n+1}+\sum_{k=0}^{n}b_k(-1)^k r^{n+1-k}$ near $(r,\bar y)=(0,(-1)^{n+1})$ is given by
\[r'=r\left(\frac{(-1)^{n}A}{n+1} +O(r)\right).\]
It points to $r=0$ when $A<0$ and $n$ even or when $A>0$ and $n$ odd, and away from $r=0$ if $A<0$ and $n$ odd or if $A>0$ and $n$ even.
\smallskip

We find no extra singularities in the positive and negative $y$-direction. The behavior of \eqref{model-Lienard1} near infinity is given in Fig. \ref{fig-m=}.
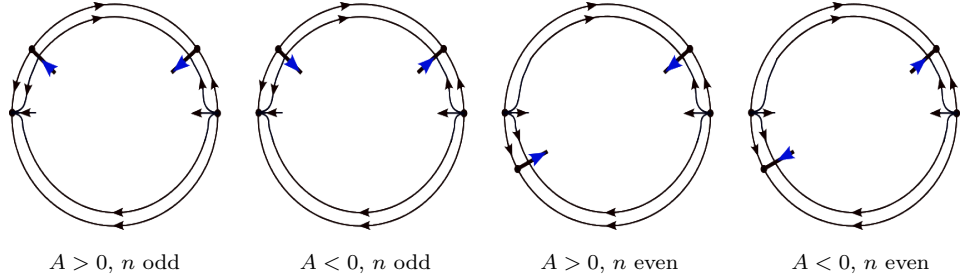
\begin{figure}[htb]
	\begin{center}
		\includegraphics[width=12.5cm,height=3cm]{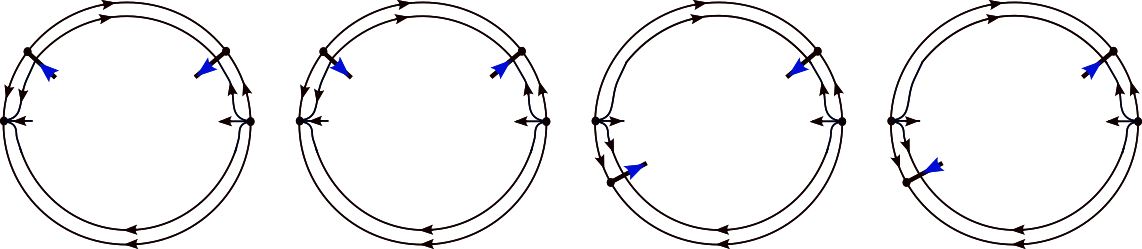}
		{\footnotesize
        \put(-247,-15){$A<0$, $n$ odd}
        \put(-341,-15){$A>0$, $n$ odd}    
        \put(-63,-15){$A<0$, $n$ even}
        \put(-157,-15){$A>0$, $n$ even}
             }
         \end{center}
 \caption{The phase portraits near infinity of \eqref{model-Lienard1}, for $m=2n+1$ and $\epsilon=0$, with indication of the slow dynamics (blue).}
	\label{fig-m=}
\end{figure}
Clearly, canard cycles with a part near infinity (Fig. \ref{fig-caninf}) are possible only if $A>0$ and $n$ odd.

In Section \ref{section-analysism=} we give the fractal analysis of \eqref{model-Lienard1} near infinity, for $A>0$ and $n$ odd, in the positive $y$-direction where we use 
\begin{equation}
x=\frac{\bar x}{r}, \ y=\frac{1}{r^{n+1}}.\nonumber
\end{equation}
System \eqref{model-Lienard1} changes, after multiplication by $r^n$, into 
\begin{equation}
\label{model-Lienardm===y+}
    \begin{vf}
        \dot{r} &=& \frac{\epsilon}{n+1}r\left( A\bar x^{2n+1}+ \sum_{k=0}^{2n}a_kr^{2n+1-k}\bar x^k  \right)  \\
        \dot{\bar x} &=& \frac{\epsilon}{n+1}\bar x \left( A\bar x^{2n+1}+ \sum_{k=0}^{2n}a_kr^{2n+1-k}\bar x^k  \right)\\
        &&+1-\left(\bar x^{n+1}+\sum_{k=0}^{n}b_kr^{n+1-k}\bar x^k \right).
    \end{vf}
\end{equation}

\subsubsection{The case $m>2n+1$}
\label{section-inf-m>}
\paragraph{$m$ is odd.} We study the dynamics of \eqref{model-Lienard1} near infinity on the Poincar\'{e}--Lyapunov disc
of degree $(1, \frac{m+1}{2})$. In the positive $x$-direction we use the transformation 
\begin{equation}
x=\frac{1}{r}, \ y=\frac{\bar y}{r^{\frac{m+1}{2}}}.\nonumber
\end{equation}
 In the coordinates $(r,\bar y)$ system \eqref{model-Lienard1} can be written as
\begin{equation}
\label{model-Lienardm>+}
    \begin{vf}
        \dot{r} &=& -r\left(\bar y-r^{\frac{m+1}{2}-n-1}\left(1+\sum_{k=0}^{n}b_kr^{n+1-k}\right) \right)   \\
        \dot{\bar y} &=&-\epsilon \left( A+ \sum_{k=0}^{m-1}a_kr^{m-k}  \right)\\
        &&-\frac{m+1}{2}\bar y \left(\bar y-r^{\frac{m+1}{2}-n-1}\left(1+\sum_{k=0}^{n}b_kr^{n+1-k}\right) \right),
    \end{vf}
\end{equation}
after multiplication by a factor $r^\frac{m-1}{2}$. When $\epsilon=r=0$, the singularity at $\bar y=0$ of \eqref{model-Lienardm>+} is linearly zero, and to  desingularize \eqref{model-Lienardm>+} we will use the following blow-up at $(r,\bar y,\epsilon)=(0,0,0)$ (see Appendix \ref{appendix}):
\begin{equation}
\label{fam-blow-up}
(r,\bar y,\epsilon)=(v\tilde r,v^{\frac{m+1}{2}-n-1}\tilde y,v^{m-2n-1}\tilde \epsilon), \ v>0, \ v\sim 0, \ \tilde r\ge 0, \ \tilde\epsilon\ge 0, \ (\tilde r,\tilde y,\tilde \epsilon)\in\mathbb S^2.
\end{equation}

When $\epsilon=0$, $\bar y=r^{\frac{m+1}{2}-n-1}\left(1+\sum_{k=0}^{n}b_kr^{n+1-k}\right)$ represents the curve of singularities of \eqref{model-Lienardm>+}, and each singularity with $r>0$ and $r\sim 0$ is normally attracting. Using asymptotic expansions in $\epsilon$, center manifolds of \eqref{model-Lienardm>+}$+0\frac{\partial}{\partial \epsilon}$ along the normally attracting portion can be written as
\begin{align}
    \bar y =& r^{\frac{m+1}{2}-n-1}\left(1+\sum_{k=0}^{n}b_kr^{n+1-k}\right)\nonumber\\
   & -\frac{A+ \sum_{k=0}^{m-1}a_kr^{m-k}}{r^{\frac{m+1}{2}-n-1}\left(n+1+\sum_{k=1}^{n}kb_kr^{n+1-k}\right)}\nonumber
\epsilon+O(\epsilon^2).
\end{align}
We substitute this for $\bar y$ in the $r$-component of \eqref{model-Lienardm>+}, and we get the slow dynamics
\begin{equation}\label{slow-dyn-m>} r'=\frac{A+ \sum_{k=0}^{m-1}a_kr^{m-k}}{r^{\frac{m+1}{2}-n-2}\left(n+1+\sum_{k=1}^{n}kb_kr^{n+1-k}\right)},
\end{equation}
after division by $\epsilon$ and $\epsilon\to 0$.
For $r>0$, it points towards $r=0$ if $A=-1$ and away from $r=0$ if $A=1$. We will use \eqref{slow-dyn-m>} in Section \ref{section-analysism>}.
\smallskip

In the negative $x$-direction we have 
\begin{equation}
x=\frac{-1}{r}, \ y=\frac{\bar y}{r^{\frac{m+1}{2}}}.\nonumber
\end{equation}
 This transformation brings system \eqref{model-Lienard1}, after multiplication by $r^\frac{m-1}{2}$, into
\begin{equation}
\label{model-Lienardm>-}
    \begin{vf}
        \dot{r} &=& r\left(\bar y-r^{\frac{m+1}{2}-n-1}\left((-1)^{n+1}+\sum_{k=0}^{n}b_k(-1)^k r^{n+1-k}\right) \right)   \\
        \dot{\bar y} &=&-\epsilon \left( -A+ \sum_{k=0}^{m-1}a_k(-1)^k r^{m-k}  \right)\\
        &&+\frac{m+1}{2}\bar y \left(\bar y-r^{\frac{m+1}{2}-n-1}\left((-1)^{n+1}+\sum_{k=0}^{n}b_k(-1)^k r^{n+1-k}\right) \right).
    \end{vf}
\end{equation}
 When $\epsilon=r=0$, the singularity at $\bar y=0$ of \eqref{model-Lienardm>-} is linearly zero. In Appendix \ref{appendix} we apply the family blow-up \eqref{fam-blow-up} at $(r,\bar y,\epsilon)=(0,0,0)$.

 When $\epsilon=0$, $\bar y=r^{\frac{m+1}{2}-n-1}\left((-1)^{n+1}+\sum_{k=0}^{n}b_k(-1)^k r^{n+1-k}\right)$ is the curve of singularities of \eqref{model-Lienardm>-}. Each singularity with $r>0$ and $r\sim 0$ is normally repelling (resp. attracting) for $n$ odd (resp. even). The slow dynamics is given by 
 \begin{equation}\label{slow-dyn-m>>} r'=\frac{-A+ \sum_{k=0}^{m-1}a_k(-1)^k r^{m-k}}{r^{\frac{m+1}{2}-n-2}\left((n+1)(-1)^{n+1}+\sum_{k=1}^{n}kb_k(-1)^kr^{n+1-k}\right)}.
\end{equation}
It is directed towards $r=0$ when $A=1$ and $n$ odd or $A=-1$ and $n$ even, and away from $r=0$ when $A=1$ and $n$ even or $A=-1$ and $n$ odd. We use \eqref{slow-dyn-m>>} in Section \ref{section-analysism>}.

We find no extra singularities in the positive and negative $y$-direction. Using the above analysis and Appendix \ref{appendix}, we find the phase portraits of \eqref{model-Lienard1} near infinity (see Fig. \ref{fig-mfamily}). It is clear that canard cycles with a part at infinity (Fig. \ref{fig-caninf}) are possible only if $A=1$ and $n$ odd.
 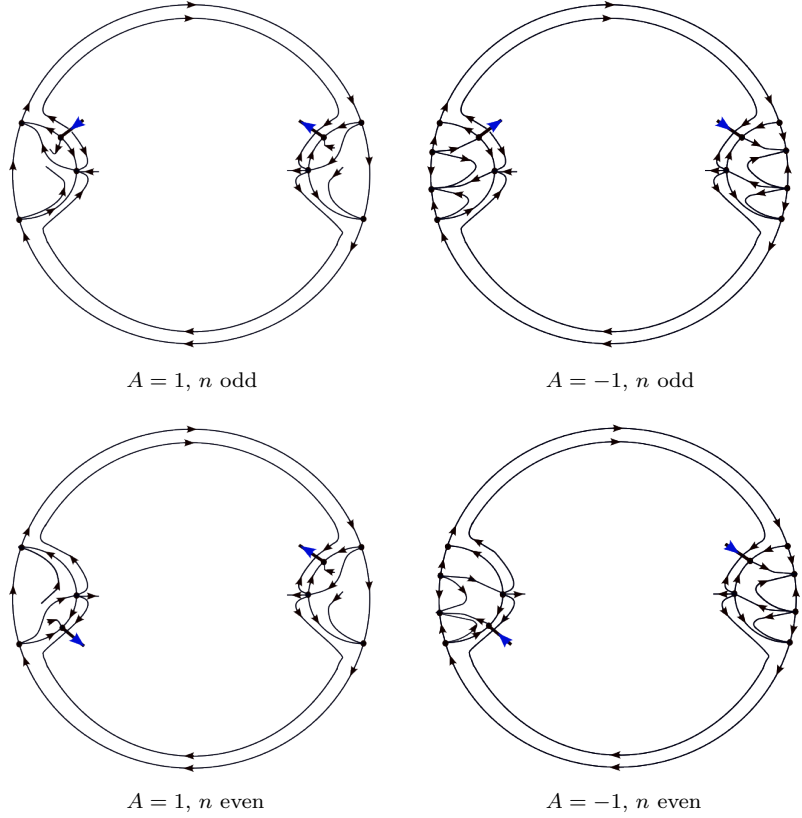
\begin{figure}[htb]
	\begin{center}
		\includegraphics[width=10.4cm,height=10.2cm]{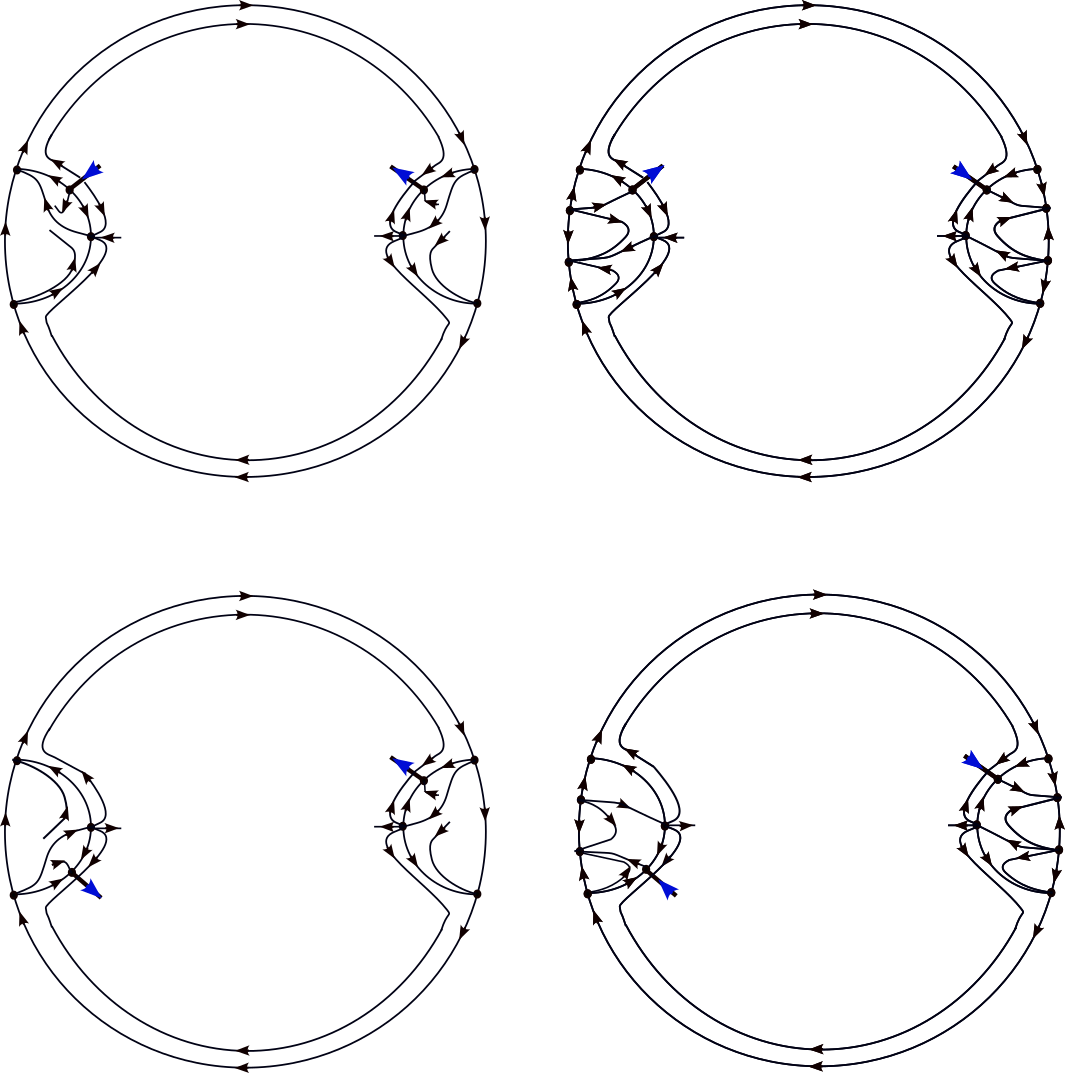}
		{\footnotesize
        \put(-252,145){$A=1$, $n$ odd}
        \put(-95,145){$A=-1$, $n$ odd}
        \put(-252,-14){$A=1$, $n$ even}
        \put(-95,-14){$A=-1$, $n$ even}
             }
         \end{center}
 \caption{The phase portraits near infinity of \eqref{model-Lienard1}, with $m>2n+1$, $m$ odd and $\epsilon=0$, and direction of the slow dynamics (blue).}
	\label{fig-mfamily}
\end{figure}

\paragraph{$m$ is even.} We study the dynamics of \eqref{model-Lienard1} near infinity on the Poincar\'{e}--Lyapunov disc
of degree $(2, m+1)$. In the positive $x$-direction we have 
\begin{equation}
x=\frac{1}{r^2}, \ y=\frac{\bar y}{r^{m+1}}.\nonumber
\end{equation}
 This transformation brings system \eqref{model-Lienard1}, after multiplication by $r^{m-1}$, into
\begin{equation}
\label{model-Lienardmeven>+}
    \begin{vf}
        \dot{r} &=& -\frac{1}{2}r\left(\bar y-r^{m-2n-1}\left(1+\sum_{k=0}^{n}b_k r^{2(n+1-k)}\right) \right)   \\
        \dot{\bar y} &=&-\epsilon \left( A+ \sum_{k=0}^{m-1}a_k r^{2(m-k)}  \right)\\
        &&-\frac{m+1}{2}\bar y \left(\bar y-r^{m-2n-1}\left(1+\sum_{k=0}^{n}b_k r^{2(n+1-k)}\right) \right).
    \end{vf}
\end{equation}
 When $\epsilon=r=0$, the singularity at $\bar y=0$ of \eqref{model-Lienardmeven>+} is linearly zero. In Appendix \ref{appendix} we will apply the following family blow-up at $(r,\bar y,\epsilon)=(0,0,0)$:
 \begin{equation}
\label{fam-blow-up--}
(r,\bar y,\epsilon)=(v\tilde r,v^{m-2n-1}\tilde y,v^{2(m-2n-1)}\tilde \epsilon), \ v>0, \ v\sim 0, \ \tilde r\ge 0, \ \tilde\epsilon\ge 0, \ (\tilde r,\tilde y,\tilde \epsilon)\in\mathbb S^2.
\end{equation}

 When $\epsilon=0$, $\bar y=r^{m-2n-1}\left(1+\sum_{k=0}^{n}b_k r^{2(n+1-k)}\right)$ is the curve of singularities of \eqref{model-Lienardmeven>+}. Each singularity with $r>0$ and $r\sim 0$ is normally  attracting. The slow dynamics is given by 
 \begin{equation} r'=\frac{1}{r^{m-2n-2}}\left(\frac{A}{2(n+1)
 }+O(r)\right).\nonumber
\end{equation}
It is directed away from $r=0$ because $A=1$. 
\smallskip

In the negative $x$-direction we have 
\begin{equation}
x=\frac{-1}{r^2}, \ y=\frac{\bar y}{r^{m+1}}.\nonumber
\end{equation}
 This transformation brings system \eqref{model-Lienard1}, after multiplication by $r^{m-1}$, into
\begin{equation}
\label{model-Lienardmeven>-}
    \begin{vf}
        \dot{r} &=& \frac{1}{2}r\left(\bar y-r^{m-2n-1}\left((-1)^{n+1}+\sum_{k=0}^{n}b_k(-1)^k r^{2(n+1-k)}\right) \right)   \\
        \dot{\bar y} &=&-\epsilon \left( A+ \sum_{k=0}^{m-1}a_k(-1)^k r^{2(m-k)}  \right)\\
        &&+\frac{m+1}{2}\bar y \left(\bar y-r^{m-2n-1}\left((-1)^{n+1}+\sum_{k=0}^{n}b_k(-1)^k r^{2(n+1-k)}\right) \right).
    \end{vf}
\end{equation}
 When $\epsilon=r=0$, the singularity at $\bar y=0$ of \eqref{model-Lienardmeven>-} is linearly zero. In Appendix \ref{appendix} we will apply the family blow-up \eqref{fam-blow-up--} at $(r,\bar y,\epsilon)=(0,0,0)$. 

 When $\epsilon=0$, $\bar y=r^{m-2n-1}\left((-1)^{n+1}+\sum_{k=0}^{n}b_k(-1)^k r^{2(n+1-k)}\right)$ is the curve of singularities of \eqref{model-Lienardmeven>-}. Each singularity with $r>0$ and $r\sim 0$ is normally repelling (resp. attracting) for $n$ odd (resp. even). The slow dynamics is given by 
 \begin{equation} r'=\frac{1}{r^{m-2n-2}}\left(\frac{A}{2(n+1)(-1)^{n+1}}+O(r)\right).\nonumber
\end{equation}
It is directed away from $r=0$ when $A=1$ and $n$ odd, and towards $r=0$ when $A=1$ and $n$ even. 
\smallskip

There are no extra singularities in the positive and negative $y$-direction. Using the above information and Appendix \ref{appendix}, we find the behavior of \eqref{model-Lienard1} near infinity (see Fig. \ref{fig-mfamilyeven}). Clearly, when $m$ is even, canard cycles with parts at infinity (Fig. \ref{fig-caninf}) are not possible.
\begin{figure}[htb]
	\begin{center}
		\includegraphics[width=9.6cm,height=4.6cm]{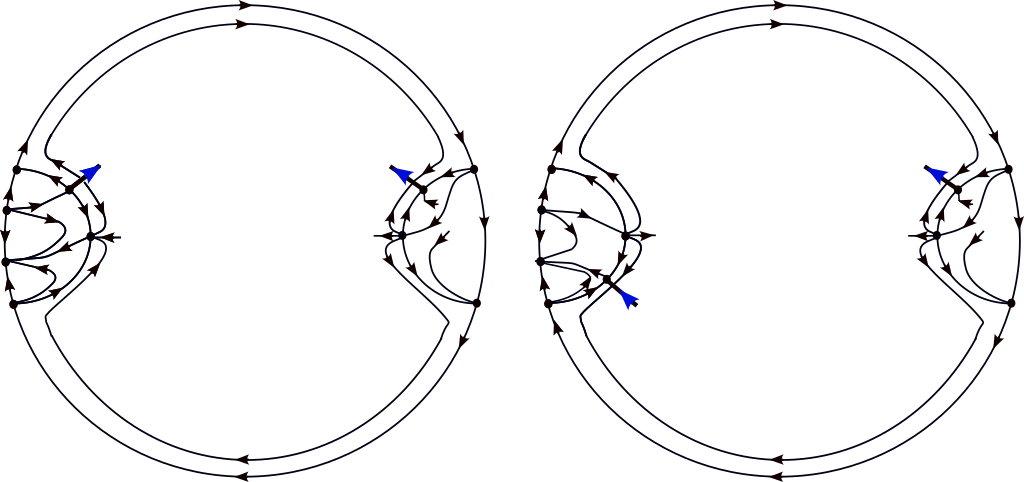}
		{\footnotesize
        \put(-233,-14){$A=1$, $n$ odd}
        \put(-88,-14){$A=1$, $n$ even}
             }
         \end{center}
 \caption{The phase portraits near infinity of \eqref{model-Lienard1}, with $m>2n+1$, $m$ even and $\epsilon=0$, and direction of the slow dynamics (blue).}
	\label{fig-mfamilyeven}
\end{figure}

\subsection{Fractal analysis near infinity}\label{sectionFAI}
In this section we present the fractal analysis for \eqref{model-Lienard1} near infinity, with $m$ and $n$ odd and $A=1$ if $m\ne 2n+1$ or $A>0$ if $m=2n+1$. This analysis  can be used, not only for proof of Theorems \ref{thm-1}--\ref{thm-3} in Section \ref{sec-proofs}, but also for the computation of the Minkowski dimension of \eqref{model-Lienard1} near infinity with a curve of singularities that is more general than the one given in Fig. \ref{fig-Motivation} and has a finite slow divergence integral (in a large compact set in the phase plane). We denote this integral by $-\tilde J$.

The main fractal results (Minkowski dimension and Minkowski nondegeneracy) are collected at the end of each section (see cases (a), (b), etc. in Sections \ref{m<frac}--\ref{section-analysism>}).

\subsubsection{The case $m<2n+1$}
\label{m<frac}
We consider system \eqref{model-Lienardm<y+} with $m<2n+1$, $m$ odd, $n$ odd and $A=1$. When $\epsilon=0$, \eqref{model-Lienardm<y+} has the curve of singularities $\bar x=\Phi_\pm(r)$ where $\Phi_+(0)=-1$, $\bar x=\Phi_+(r)$ is normally repelling, $\Phi_-(0)=1$, $\bar x=\Phi_-(r)$ is normally attracting and
\begin{equation}\label{Phipm}
    1-\Phi_\pm(r)^{n+1}-\sum_{k=0}^{n}b_kr^{n+1-k}\Phi_\pm(r)^k =0.
\end{equation}
If we substitute $\Phi_\pm(r)$ for $\bar x$ in the $r$-component of \eqref{model-Lienardm<y+}, then we get the slow dynamics along $\bar x=\Phi_\pm(r)$ (after division by $\epsilon$)
\begin{equation}
    \label{SD-bitno1}
    \frac{dr}{d\tau} = \frac{r^{2n+2-m}}{n+1}\left( \Phi_\pm(r)^m+ \sum_{k=0}^{m-1}a_kr^{m-k}\Phi_\pm(r)^k  \right).
\end{equation}
The slow divergence integral associated to $\bar x=\Phi_\pm(r)$ is given by 
\begin{equation}\label{bitno2}
    J_\pm(r)=-(n+1)\int_r^{\tilde r}  \frac{(n+1)\Phi_\pm(s)^{n}+\sum_{k=1}^{n}k b_ks^{n+1-k}\Phi_\pm(s)^{k-1}}{s^{2n+2-m}\left( \Phi_\pm(s)^m+ \sum_{k=0}^{m-1}a_ks^{m-k}\Phi_\pm(s)^k  \right)}ds <0
\end{equation}
where $\tilde r>0$ is small and fixed and $r\in]0,\tilde r[$. Notice that the divergence of \eqref{model-Lienardm<y+} on $\bar x=\Phi_\pm(r)$ is given by $-(n+1)\Phi_\pm(r)^{n}-\sum_{k=1}^{n}k b_kr^{n+1-k}\Phi_\pm(r)^{k-1}$ when $\epsilon=0$, while $d\tau$ is obtained from \eqref{SD-bitno1}. Since $m<2n+1$, we have that $J_\pm(r)$ monotonically go to $-\infty$ when $r\to 0$.
\smallskip

Let $\tilde J$ be a real constant. Assume that a sequence $(r_l)_{l\in\mathbb{N}}$ generated by
\begin{equation}\label{bitno3}
  J_-(r_l)-J_+(r_{l+1})=\tilde J , \  \ r_0\in]0,\tilde r[, \ l\in\mathbb N, 
\end{equation}
is well-defined, i.e. $r_l\to 0$ (monotonically) as $l\to \infty$. The existence of such a sequence $(r_l)_{l\in\mathbb{N}}$, for all $r_0>0$ small enough, will become clear later in this section (see cases (a)--(h)). Using \eqref{bitno3} we have 
\begin{equation}\label{bitno4}
  J_-(r_l)-J_+(r_{l})=\tilde J +J_+(r_{l+1}) -J_+(r_{l}), \ l\in\mathbb N.
\end{equation}

First, we study $J_-(r_l)-J_+(r_{l})$ in \eqref{bitno4}. We need the following result.
\begin{lemma}\label{label-LemmaPhi} We have
\begin{equation}\label{expansionPhi}
    \Phi_-(r)=-\Phi_+(r)-\frac{2}{n+1}\sum_{j=0}^{\frac{n-1}{2}}b_{2j+1}r^{n-2j}(1+O(r))
    \end{equation}
where the functions $\Phi_\pm$ are given in \eqref{Phipm}.
\end{lemma}
\begin{proof}
We have for $r\sim 0$
\begin{align}\label{alig-Phi}
  \Phi_\pm(r)&=\mp\left( 1-\sum_{k=0}^{n}b_kr^{n+1-k}\Phi_\pm(r)^k \right)^{\frac{1}{n+1}} \nonumber \\
  &=\mp 1\pm \frac{1}{n+1} \left(\sum_{k=0}^{n}b_kr^{n+1-k}\Phi_\pm(r)^k\right)(1+O(r))\nonumber\\
  &=\mp 1\pm \frac{1}{n+1} \sum_{k=0}^{n}b_kr^{n+1-k}\left((\mp 1)^k+O(r)\right).
\end{align}
In the first step we used \eqref{Phipm}, in the second step the binomial series $(1-s)^\frac{1}{n+1}=1-\frac{1}{n+1}s+\dots$, for $s\sim 0$, and in the last step we used $\Phi_+(0)=-1$ and $\Phi_-(0)=1$. Since $\Phi_\pm$ are unique (the Implicit Function Theorem), from \eqref{Phipm} it follows that $\Phi_-(r)=-\Phi_+(r)$ when $b_1=b_3=\dots=b_n=0$. It can be easily seen that this and \eqref{alig-Phi} imply \eqref{expansionPhi}.
\end{proof}
Using Lemma \ref{label-LemmaPhi} and writing $\Phi_\pm=\Phi_\pm(s)$ we get
\begin{align}\label{alig-J}
 &\frac{(n+1)\Phi_-^{n}+\sum_{k=1}^{n}k b_ks^{n+1-k}\Phi_-^{k-1}}{ \Phi_-^m+ \sum_{k=0}^{m-1}a_ks^{m-k}\Phi_-^k}  \nonumber\\
 & \qquad = \frac{(n+1)(-\Phi_+)^{n}+\sum_{k=1}^{n}k b_ks^{n+1-k}(-\Phi_+)^{k-1}}{ (-\Phi_+)^m+ \sum_{k=0}^{m-1}a_ks^{m-k}(-\Phi_+)^k}\nonumber\\
 & \qquad \ \ \ +\sum_{j=0}^{\frac{n-1}{2}}b_{2j+1}s^{n-2j}(2(m-n)+O(s))\nonumber\\
  & \qquad = \frac{-(n+1)\Phi_+^{n}-\sum_{k=1}^{n}k b_ks^{n+1-k}\Phi_+^{k-1}+2\sum_{j=0}^{\frac{n-1}{2}}(2j+1)b_{2j+1}s^{n-2j}\Phi_+^{2j}}{ -\Phi_+^m- \sum_{k=0}^{m-1}a_ks^{m-k}\Phi_+^k+2\sum_{j=0}^{\frac{m-1}{2}}a_{2j}s^{m-2j}\Phi_+^{2j}}\nonumber\\
 & \qquad \ \ \ +\sum_{j=0}^{\frac{n-1}{2}}b_{2j+1}s^{n-2j}(2(m-n)+O(s))\nonumber\\
 & \qquad = \frac{(n+1)\Phi_+^{n}+\sum_{k=1}^{n}k b_ks^{n+1-k}\Phi_+^{k-1}}{ \Phi_+^m+ \sum_{k=0}^{m-1}a_ks^{m-k}\Phi_+^k}-\sum_{j=0}^{\frac{m-1}{2}}a_{2j}s^{m-2j}(2(n+1)+O(s))\nonumber\\
  & \qquad \ \ \ +\sum_{j=0}^{\frac{n-1}{2}}b_{2j+1}s^{n-2j}(2(m-n+2j+1)+O(s)).
\end{align}
In the first step we used \eqref{expansionPhi} and the expansion of order $1$ in powers of the sum in \eqref{expansionPhi} at zero, and in the last step we used the expansion of order $1$ in powers of the sums $2\sum_{j=0}^{\frac{n-1}{2}}$ and $2\sum_{j=0}^{\frac{m-1}{2}}$ at zero.
From \eqref{bitno2} and \eqref{alig-J} it follows that
\begin{align}\label{alig-JJJ}
 J_-(r_l)-&J_+(r_{l})=-(n+1)\int_{r_l}^{\tilde r}  \frac{1}{s^{2n+2-m}}\Bigg(  -\sum_{j=0}^{\frac{m-1}{2}}a_{2j}s^{m-2j}(2(n+1)+O(s))\nonumber \\ 
 & \ \  +\sum_{j=0}^{\frac{n-1}{2}}b_{2j+1}s^{n-2j}(2(m-n+2j+1)+O(s)) \Bigg) ds\nonumber \\
 &= -(n+1)\Bigg(  -\sum_{j=0}^{\frac{m-1}{2}}a_{2j}\int_{r_l}^{\tilde r}s^{2(m-n-j-1)}(2(n+1)+O(s))ds\nonumber \\ 
 & \ \ +\sum_{j=0}^{\frac{n-1}{2}}b_{2j+1}\int_{r_l}^{\tilde r}s^{m-n-2j-2}(2(m-n+2j+1)+O(s))ds \Bigg)\nonumber \\
 &= -\sum_{j=0}^{\frac{m-1}{2}}a_{2j}\frac{2(n+1)^2}{2m-2n-2j-1} r_{l}^{2m-2n-2j-1}(1+o(1))\nonumber \\ 
 & \ \ +\sum_{j=0}^{\frac{n-1}{2}}b_{2j+1}\frac{2(n+1)(m-n+2j+1)}{m-n-2j-1} r_l^{m-n-2j-1}(1+o(1))+\bar J
\end{align}
where the $o(1)$-terms tend to zero as $r_l\to 0$ and $\bar J$ is a constant independent of $r_l$. In the last step we used the fact that $2(m-n-j-1)$ and $m-n-2j-2$ are even, thus $\ne -1$ ($m,n$ are odd).

Using \eqref{bitno2}, the term $J_+(r_{l+1}) -J_+(r_{l})$ in \eqref{bitno4} can be written as
\begin{equation}\label{desno}
J_+(r_{l+1}) -J_+(r_{l})=-(n+1)^2\int_{r_{l+1}}^{r_l}\frac{1}{s^{2n+2-m}}(1+O(s))ds.
\end{equation}
If we use \eqref{bitno4}, \eqref{alig-JJJ} and the substitution $s=r_l t$ in the integral in \eqref{desno}, we get 
\begin{align}\label{desno-lijevo}
 -(n&+1)^2\int_{\frac{r_{l+1}}{r_l}}^{1}\frac{1}{t^{2n+2-m}}(1+O(r_lt))dt\nonumber\\
 &= -\sum_{j=0}^{\frac{m-1}{2}}a_{2j}\frac{2(n+1)^2}{2m-2n-2j-1} r_{l}^{m-2j}(1+o(1))\nonumber \\ 
 & \ \ +\sum_{j=0}^{\frac{n-1}{2}}b_{2j+1}\frac{2(n+1)(m-n+2j+1)}{m-n-2j-1} r_l^{n-2j}(1+o(1))+r_l^{2n+1-m}(\bar J-\tilde J).
\end{align}
Since the right-hand side of \eqref{desno-lijevo} tends to zero as $l\to \infty$ (note that $m<2n+1$), we have that $\frac{r_{l+1}}{r_l}\to 1$ as $l\to\infty$. This and the fact that $\frac{\rho}{r_l^{2n+2-m}}\le  \frac{1}{s^{2n+2-m}}(1+O(s))\le \frac{1}{\rho r_{l+1}^{2n+2-m}}$ for all $s\in [r_{l+1},r_l]$, with $\rho>0$ small enough, imply that the integral in \eqref{desno} has the following property:
\begin{equation}
\label{asimptotica}
r_l^{2n+2-m}\int_{r_{l+1}}^{r_l}\frac{1}{s^{2n+2-m}}(1+O(s))ds\simeq r_l-r_{l+1}, \ \ l\to\infty.
\end{equation}

Let's recall that $j_a$, $j_b$, $a_e$ and $b_o$ are defined in Section \ref{sec-Motivation} before Theorem \ref{thm-1}. In cases (a)--(g) below we assume that at least one of $j_a$ 
 and $j_b$ is well-defined.

\paragraph{(a) the case ($n-2j_b<m-2j_a$ or $a_e=0$) and $m-n-2j_b-1< 0$.} Since $n-2j_b<m-2j_a$ or $a_e=0$, \eqref{alig-JJJ} implies that 
\begin{equation}\label{eqcase1}
J_-(r)-J_+(r)=b_{2j_b+1}\frac{2(n+1)(m-n+2j_b+1)}{m-n-2j_b-1} r^{m-n-2j_b-1}(1+o(1))+\bar J,
\end{equation}
where $o(1)\to 0$ as $r\to 0$. Assume first that $b_{2j_b+1}(m-n+2j_b+1)>0$. From \eqref{eqcase1} and $m-n-2j_b-1< 0$ it follows that $J_-(r)-(J_+(r)+\bar J)<0$, for all $r>0$ small enough, and that $J_-(r)-(J_+(r)+\bar J)\to -\infty$ as $r\to 0$. Now, it is clear that $(r_l)_{l\in\mathbb{N}}$ generated by $J_-(r_l)-(J_+(r_{l+1})+\bar J)=\tilde J-\bar J$ (or, equivalently, by \eqref{bitno3}) is well-defined for each $r_0>0$ small enough, i.e. it tends monotonically to zero as $l\to +\infty$. We also used the fact that $J_\pm(r) \to -\infty$ as $r\to 0$. Using \eqref{bitno4}, \eqref{desno}, \eqref{asimptotica}, \eqref{eqcase1} and $m-n-2j_b-1< 0$, finally we get 
\begin{equation}\label{finally11}
   r_l-r_{l+1}  \simeq r_l^{n+1-2j_b}, \ l\to \infty.
\end{equation}
Since $n+1-2j_b>1$ (note that $j_b\le \frac{n-1}{2}$), \eqref{finally11} and \cite[Theorem 1]{EZZ} imply that the sequence $(r_l)_{l\in\mathbb{N}}$ is Minkowski nondegenerate,
\[\dim_B(r_l)_{l\in\mathbb{N}}=1-\frac{1}{n+1-2j_b}=\frac{n-2j_b}{n+1-2j_b}\]
and these results are independent of the choice of $r_0$.

If $b_{2j_b+1}(m-n+2j_b+1)<0$, then $(r_l)_{l\in\mathbb{N}}$ is generated by $J_-(r_{l+1})-J_+(r_{l})=\tilde J$, instead of \eqref{bitno3}. Using similar computations we get \eqref{finally11} and the same Minkowski dimension as above.

\paragraph{(b) the case ($n-2j_b<m-2j_a$ or $a_e=0$) and $m-n-2j_b-1>0$.} Since $n-2j_b<m-2j_a$ or $a_e=0$, we have \eqref{eqcase1}. Assume that $b_{2j_b+1}<0$. From \eqref{eqcase1} and $m-n-2j_b-1>0$ it follows that $J_-(r)-(J_+(r)+\bar J)<0$, for all $r>0$ small enough, and that $J_-(r)-(J_+(r)+\bar J)\to 0$ as $r\to 0$. This implies that, for $\tilde J-\bar J\ge 0$, $(r_l)_{l\in\mathbb{N}}$ generated by $J_-(r_l)-(J_+(r_{l+1})+\bar J)=\tilde J-\bar J$ (i.e. by \eqref{bitno3}) is well-defined for each $r_0>0$ small enough, i.e. it tends monotonically to zero as $l\to +\infty$. If $\tilde J-\bar J> 0$, then \eqref{bitno4}, \eqref{desno}, \eqref{asimptotica}, \eqref{eqcase1} and $m-n-2j_b-1> 0$ give 
\begin{equation}\label{finally111}
   r_l-r_{l+1}  \simeq r_l^{2n+2-m}, \ l\to \infty.
\end{equation}
Using \eqref{finally111}, \cite[Theorem 1]{EZZ} and the fact that $2n+2-m>1$ we have that $(r_l)_{l\in\mathbb{N}}$ is Minkowski nondegenerate,
\begin{equation}\label{boxdim4444}
\dim_B(r_l)_{l\in\mathbb{N}}=1-\frac{1}{2n+2-m}=\frac{2n+1-m}{2n+2-m}
\end{equation}
and these results are independent of the choice of $r_0$. If $\tilde J-\bar J=0$, then \eqref{bitno4}, \eqref{desno}, \eqref{asimptotica} and \eqref{eqcase1} imply \eqref{finally11} and, thus, the same Minkowski dimension of $(r_l)_{l\in\mathbb{N}}$ like in case (a).

If $\tilde J-\bar J<0$, then $(r_l)_{l\in\mathbb{N}}$ is generated by $J_-(r_{l+1})-J_+(r_{l})=\tilde J$, for each $r_0>0$ small enough. The fractal analysis of $(r_l)_{l\in\mathbb{N}}$ in this case is similar to the fractal analysis in case $\tilde J-\bar J> 0$. We obtain \eqref{finally111} and \eqref{boxdim4444}.

Assume now that $b_{2j_b+1}>0$. This is similar to the case $b_{2j_b+1}<0$. If $\tilde J-\bar J>0$, then $(r_l)_{l\in\mathbb{N}}$ is generated by \eqref{bitno3}, for each $r_0>0$ small enough, and we have \eqref{finally111} and \eqref{boxdim4444}. If $\tilde J-\bar J\le 0$, then $(r_l)_{l\in\mathbb{N}}$ is generated by $J_-(r_{l+1})-J_+(r_{l})=\tilde J$, for each $r_0>0$ small enough. When $\tilde J-\bar J< 0$, we have \eqref{finally111} and \eqref{boxdim4444}, and when $\tilde J-\bar J=0$, we have \eqref{finally11} and the same Minkowski dimension of $(r_l)_{l\in\mathbb{N}}$ like in case (a).

\paragraph{(c) the case ($n-2j_b>m-2j_a$ or $b_o=0$) and $2m-2n-2j_a-1<0$.} The fractal analysis in this case is analogous to the fractal analysis in case (a). Since $n-2j_b>m-2j_a$ or $b_o=0$, \eqref{alig-JJJ} implies that 
\begin{equation}\label{eqcase1234}
J_-(r)-J_+(r)=-a_{2j_a}\frac{2(n+1)^2}{2m-2n-2j_a-1} r^{2m-2n-2j_a-1}(1+o(1))+\bar J,
\end{equation}
where $o(1)\to 0$ as $r\to 0$. Assume first that $a_{2j_a}<0$. It follows that $(r_l)_{l\in\mathbb{N}}$ generated  by \eqref{bitno3} is well-defined for each $r_0>0$ small enough (i.e. it tends monotonically to zero as $l\to +\infty$). Using \eqref{bitno4}, \eqref{desno}, \eqref{asimptotica}, \eqref{eqcase1234} and $2m-2n-2j_a-1<0$ we get 
\begin{equation}\label{finally11234}
   r_l-r_{l+1}  \simeq r_l^{m+1-2j_a}, \ l\to \infty.
\end{equation}
Since $m+1-2j_a>1$ ($j_a\le \frac{m-1}{2}$), \eqref{finally11234} and \cite[Theorem 1]{EZZ} imply that the sequence $(r_l)_{l\in\mathbb{N}}$ is Minkowski nondegenerate,
\[\dim_B(r_l)_{l\in\mathbb{N}}=1-\frac{1}{m+1-2j_a}=\frac{m-2j_a}{m+1-2j_a}\]
and these results don't depend on the choice of $r_0$.

If $a_{2j_a}>0$, then $(r_l)_{l\in\mathbb{N}}$ is generated by $J_-(r_{l+1})-J_+(r_{l})=\tilde J$, for each $r_0>0$ small enough. Using similar computations we get \eqref{finally11234} and the same Minkowski dimension as above.

\paragraph{(d) the case ($n-2j_b>m-2j_a$ or $b_o=0$) and $2m-2n-2j_a-1>0$.} The fractal analysis in this case is analogous to the fractal analysis in case (b). We use \eqref{eqcase1234}. Assume that $a_{2j_a}>0$. For $\tilde J-\bar J\ge 0$, $(r_l)_{l\in\mathbb{N}}$ is generated by \eqref{bitno3} for each $r_0>0$ small enough. If $\tilde J-\bar J> 0$, then we have  \eqref{finally111} and \eqref{boxdim4444}.
 If $\tilde J-\bar J=0$, then we have \eqref{finally11234} and, thus, the same Minkowski dimension of $(r_l)_{l\in\mathbb{N}}$ like in case (c). If $\tilde J-\bar J<0$, then $(r_l)_{l\in\mathbb{N}}$ is generated by $J_-(r_{l+1})-J_+(r_{l})=\tilde J$, for each $r_0>0$ small enough. We obtain \eqref{finally111} and \eqref{boxdim4444}.

 Assume now that $a_{2j_a}<0$. If $\tilde J-\bar J>0$, then $(r_l)_{l\in\mathbb{N}}$ is generated by \eqref{bitno3}, for each $r_0>0$ small enough, and we have \eqref{finally111} and \eqref{boxdim4444}. If $\tilde J-\bar J\le 0$, then $(r_l)_{l\in\mathbb{N}}$ is generated by $J_-(r_{l+1})-J_+(r_{l})=\tilde J$, for each $r_0>0$ small enough. When $\tilde J-\bar J< 0$, we have \eqref{finally111} and \eqref{boxdim4444}, and when $\tilde J-\bar J=0$, we have \eqref{finally11234} and the same Minkowski dimension of $(r_l)_{l\in\mathbb{N}}$ like in case (c).

 In cases (e), (f) and (g) we write $C:=b_{2j_b+1}(m-n+2j_b+1)-a_{2j_a}(n+1)$.

 \paragraph{(e) the case $n-2j_b=m-2j_a$, $C\ne 0$ and $m-n-2j_b-1< 0$.} The fractal analysis in this case is analogous to the fractal analysis in case (a). We have 
 \begin{equation}\label{eqcasezadnjiiii}
J_-(r)-J_+(r)=C\frac{2(n+1)}{m-n-2j_b-1} r^{m-n-2j_b-1}(1+o(1))+\bar J.
\end{equation}
If $C>0$ (resp. $C<0$), then we have the same analysis and results as in case (a) with $b_{2j_b+1}(m-n+2j_b+1)>0$ (resp. $b_{2j_b+1}(m-n+2j_b+1)<0$).

\paragraph{(f) the case $n-2j_b=m-2j_a$, $C\ne 0$ and $m-n-2j_b-1>0$.} The fractal analysis in this case is analogous to the fractal analysis in case (b). We have \eqref{eqcasezadnjiiii}. If $C<0$ (resp. $C>0$), then we have the same analysis and results as in case (b) with $b_{2j_b+1}<0$ (resp. $b_{2j_b+1}>0$). 

 \paragraph{(g) the case $n-2j_b=m-2j_a$ and $C= 0$.} This is a topic of further study.

\paragraph{(h) the case $a_e=b_o=0$.} From \eqref{alig-JJJ} it follows that $J_-(r)=J_+(r)$. This implies that $(r_l)_{l\in\mathbb{N}}$, generated by \eqref{bitno3} with $\tilde J=0$, is a constant sequence (hence its Minkowski dimension is $0$). If $\tilde J>0$ (resp. $\tilde J<0$), then $(r_l)_{l\in\mathbb{N}}$, generated by \eqref{bitno3} (resp. by $J_-(r_{l+1})-J_+(r_{l})=\tilde J$) tends monotonically to $0$, for each small $r_0>0$. If $\tilde J>0$, then using \eqref{bitno4}, \eqref{desno} and \eqref{asimptotica} we get \eqref{finally111} and \eqref{boxdim4444}. We obtain the same result when $\tilde J<0$.

\subsubsection{The case $m=2n+1$} 
\label{section-analysism=}
We consider system \eqref{model-Lienardm===y+} with $n$ odd and $A>0$. We use the notation from Section \ref{m<frac}. For $\epsilon=0$, we denote by $\bar x=\Phi_-(r)$ with $\Phi_-(0)=1$ (resp. $\bar x=\Phi_+(r)$ with $\Phi_+(0)=-1$) the normally attracting (resp. repelling) curve of singularities of \eqref{model-Lienardm===y+}. $\Phi_-(r)$ and $\Phi_+(r)$ satisfy \eqref{Phipm} and have the property \eqref{expansionPhi} in Lemma \ref{label-LemmaPhi}. The slow dynamics along $\bar x=\Phi_\pm(r)$ is given by $\frac{dr}{d\tau} = \frac{r}{n+1}\left( A\Phi_\pm(r)^{2n+1}+ \sum_{k=0}^{2n}a_kr^{2n+1-k}\Phi_\pm(r)^k  \right)$,
and the slow divergence integral associated to $\bar x=\Phi_\pm(r)$ is given by 
\begin{equation}\label{bitno2====}
    J_\pm(r)=-(n+1)\int_r^{\tilde r}  \frac{(n+1)\Phi_\pm(s)^{n}+\sum_{k=1}^{n}k b_ks^{n+1-k}\Phi_\pm(s)^{k-1}}{s\left( A\Phi_\pm(s)^{2n+1}+ \sum_{k=0}^{2n}a_ks^{2n+1-k}\Phi_\pm(s)^k  \right)}ds <0
\end{equation}
where $\tilde r>0$ is small and fixed and $r\in]0,\tilde r[$ (see Section \ref{m<frac}). Assume that a sequence $(r_l)_{l\in\mathbb{N}}$, defined by \eqref{bitno3}, or equivalently by \eqref{bitno4}, monotonically tends to zero as $l\to \infty$. Later in this section it will be clear when this is possible (for each initial point $r_0>0$ small enough).

Using the same steps as in \eqref{alig-J} and $\Phi_\pm=\Phi_\pm(s)$ we get
\begin{align}
 &\frac{(n+1)\Phi_-^{n}+\sum_{k=1}^{n}k b_ks^{n+1-k}\Phi_-^{k-1}}{A \Phi_-^{2n+1}+ \sum_{k=0}^{2n}a_ks^{2n+1-k}\Phi_-^k}  \nonumber\\
 & \qquad = \frac{(n+1)\Phi_+^{n}+\sum_{k=1}^{n}k b_ks^{n+1-k}\Phi_+^{k-1}}{A \Phi_+^{2n+1}+ \sum_{k=0}^{2n}a_ks^{2n+1-k}\Phi_+^k}-\sum_{j=0}^{n}a_{2j}s^{2n+1-2j}(\frac{2}{A^2}(n+1)+O(s))\nonumber\\
  & \qquad \ \ \ +\sum_{j=0}^{\frac{n-1}{2}}b_{2j+1}s^{n-2j}(\frac{2}{A}(n+2j+2)+O(s))\nonumber,
\end{align}
and then, using \eqref{bitno2====}, 
\begin{align}\label{alig-JJJ=====}
 J_-(r_l)-&J_+(r_{l})=\nonumber \\
 & -\sum_{j=0}^{n}a_{2j}\frac{2(n+1)^2}{A^2(2n+1-2j)} r_{l}^{2n+1-2j}(1+o(1))\nonumber \\ & \ \ +\sum_{j=0}^{\frac{n-1}{2}}b_{2j+1}\frac{2(n+1)(n+2j+2)}{A(n-2j)} r_l^{n-2j}(1+o(1))+\bar J
\end{align}
where $o(1)$-functions tend to $0$ as $r_l\to 0$ and $\bar J$ is a constant independent of $r_l$.

On the other hand, we have 
\begin{equation}\label{desno=====}
J_+(r_{l+1}) -J_+(r_{l})=-\frac{(n+1)^2}{A}\int_{r_{l+1}}^{r_l}\frac{1}{s}(1+O(s))ds.
\end{equation}
If we use \eqref{bitno4}, \eqref{alig-JJJ=====} and the substitution $s=r_l t$ in the integral in \eqref{desno=====}, we obtain 
\begin{align}\label{desno-lijevo=====}
 \lim_{l\to\infty}\int_{\frac{r_{l+1}}{r_l}}^{1}\frac{1}{t}(1+O(r_lt))dt=\frac{A}{(n+1)^2}(\tilde J-\bar J).\nonumber
\end{align}
For $\tilde J-\bar J\ge 0$, this implies that
\begin{equation}
\label{limit-m=}
\lim_{l\to\infty}\frac{r_{l+1}}{r_l}=e^{-\frac{A}{(n+1)^2}(\tilde J-\bar J)}\in ]0,1].
\end{equation}
When $\tilde J-\bar J= 0$, we will need the following property of the integral in \eqref{desno=====}:
\begin{equation}
\label{asimptotica=====}
r_l\int_{r_{l+1}}^{r_l}\frac{1}{s}(1+O(s))ds\simeq r_l-r_{l+1}, \ \ l\to\infty.
\end{equation}
\eqref{asimptotica=====} can be proved in the same fashion  as \eqref{asimptotica} using \eqref{limit-m=}.

\begin{remark}\label{remarkbitno} From \eqref{bitno2====} it follows that $J_\pm(r)$ monotonically tend to $-\infty$ as $r\to 0$ and from \eqref{alig-JJJ=====} it follows that $J_-(r)-(J_+(r)+\bar J)\to 0$ as $r\to 0$. 
\end{remark}

\paragraph{(a) the case $\tilde J-\bar J\ne 0$.} Assume first that $\tilde J-\bar J>0$. Remark \ref{remarkbitno} implies that $(r_l)_{l\in\mathbb{N}}$, defined by $J_-(r_l)-(J_+(r_{l+1})+\bar J)=\tilde J-\bar J$ (i.e. by \eqref{bitno3}), tends monotonically to zero as $l\to +\infty$, for each sufficiently small initial point $r_0>0$. Since $\tilde J-\bar J>0$, from \eqref{limit-m=} it follows that there exists $\lambda\in ]0,1[$ and a constant $C>0$ such that $0<r_l\le C \lambda^l$ for all $l$ (i.e. $(r_l)_{l\in\mathbb{N}}$ converges exponentially to zero). Following \cite[Lemma 1]{EZZ} we have that $\dim_B(r_l)_{l\in\mathbb{N}}=0$.

If $\tilde J-\bar J<0$, Remark \ref{remarkbitno} implies that $(r_l)_{l\in\mathbb{N}}$, defined by $J_-(r_{l+1})-J_+(r_{l})=\tilde J$, tends monotonically to zero as $l\to +\infty$, for each small $r_0>0$. It can be proved in a similar way that $\dim_B(r_l)_{l\in\mathbb{N}}=0$.

\paragraph{(b) the case ($n-2j_b<2n+1-2j_a$ or $a_e=0$) and $\tilde J-\bar J=0$.} Since $n-2j_b<2n+1-2j_a$ or $a_e=0$, \eqref{alig-JJJ=====} implies that 
\begin{equation}
J_-(r)-J_+(r)=b_{2j_b+1}\frac{2(n+1)(n+2j_b+2)}{A(n-2j_b)} r^{n-2j_b}(1+o(1))+\bar J,\nonumber
\end{equation}
where $o(1)\to 0$ as $r\to 0$. This, together with \eqref{bitno4}, \eqref{desno=====}, \eqref{asimptotica=====} and Remark \ref{remarkbitno}, implies that we have \eqref{finally11}, with the Minkowski dimension given after \eqref{finally11}. See the cases ($b_{2j_b+1}<0,\tilde J-\bar J=0$) and ($b_{2j_b+1}>0,\tilde J-\bar J=0$) in case (b) in Section \ref{m<frac}.

\paragraph{(c) the case ($n-2j_b>2n+1-2j_a$ or $b_o=0$) and $\tilde J-\bar J=0$.} Since $n-2j_b>2n+1-2j_a$ or $b_o=0$, \eqref{alig-JJJ=====} implies that 
\begin{equation}
J_-(r)-J_+(r)=-a_{2j_a}\frac{2(n+1)^2}{A^2(2n+1-2j_a)} r^{2n+1-2j_a}(1+o(1))+\bar J,\nonumber
\end{equation}
where $o(1)\to 0$ as $r\to 0$. This, together with \eqref{bitno4}, \eqref{desno=====}, \eqref{asimptotica=====} and Remark \ref{remarkbitno}, implies that we have \eqref{finally11234} and the Minkowski dimension given after \eqref{finally11234} with $m=2n+1$. See the cases ($a_{2j_a}>0,\tilde J-\bar J=0$) and ($a_{2j_a}<0,\tilde J-\bar J=0$) in case (d) in Section \ref{m<frac}.

\paragraph{(d) the case $n-2j_b=2n+1-2j_a$, $C:=b_{2j_b+1}(n+2j_b+2)-a_{2j_a}\frac{n+1}{A}\ne 0$ and $\tilde J-\bar J=0$.} From \eqref{alig-JJJ=====} it follows that 
\begin{equation}
J_-(r)-J_+(r)=C\frac{2(n+1)}{A(n-2j_b)} r^{n-2j_b}(1+o(1))+\bar J,\nonumber
\end{equation}
where $o(1)\to 0$ as $r\to 0$. This, together with \eqref{bitno4}, \eqref{desno=====}, \eqref{asimptotica=====} and Remark \ref{remarkbitno}, implies that we have \eqref{finally11}, with the Minkowski dimension given after \eqref{finally11}. This is similar to case (b) in this section. 

\paragraph{(e) the case $n-2j_b=2n+1-2j_a$, $C=0$ and $\tilde J-\bar J=0$.} This is a topic of further study.

\paragraph{(f) the case $a_e=b_o=0$ and $\tilde J-\bar J=0$.} Here we deal with constant sequences $(r_l)_{l\in\mathbb{N}}$ (with trivial Minkowski dimension).

\subsubsection{The case $m>2n+1$}
\label{section-analysism>}
In this section we focus on the fractal analysis of \eqref{model-Lienard1} near infinity, with $m>2n+1$, $m$ odd, $n$ odd and $A=1$. Consider system \eqref{model-Lienardm>+} (resp. system \eqref{model-Lienardm>-}) from Section \ref{section-inf-m>}. The slow dynamics along the curve of singularities \[\bar y=r^{\frac{m+1}{2}-n-1}\left(1+\sum_{k=0}^{n}b_kr^{n+1-k}\right)\]\[\left(\text{resp. } \bar y=r^{\frac{m+1}{2}-n-1}\left(1+\sum_{k=0}^{n}b_k(-1)^k r^{n+1-k}\right)\right)\] of \eqref{model-Lienardm>+} (resp. \eqref{model-Lienardm>-}) is given in \eqref{slow-dyn-m>} (resp. \eqref{slow-dyn-m>>}), and the slow divergence integral
associated to the curve of singularities is given by
\begin{equation}\label{SDI999-}
\bar J_-( r)=-\int_{0}^{ r}\frac{s^{m-2n-2}\left(n+1+\sum_{k=1}^{n}kb_ks^{n+1-k}\right)^2}{1+ \sum_{k=0}^{m-1}a_ks^{m-k}}ds<0
\end{equation}
\begin{equation}\label{SDI999+}
\left(\text{resp. }\bar J_+( r)=-\int_{0}^{ r}\frac{s^{m-2n-2}\left(n+1+\sum_{k=1}^{n}kb_k(-1)^ks^{n+1-k}\right)^2}{1+ \sum_{k=0}^{m-1}a_k(-1)^{k+1}s^{m-k}}ds<0\right)
\end{equation}
where $ r>0$ is small. It is clear that $\bar J_\pm$ monotonically tend to $0$ as $r\to 0$.
Let's recall that in Section \ref{section-inf-m>} ($m$ odd)  we use the Poincar\'{e}--Lyapunov compactification
of degree $(1, \frac{m+1}{2})$ and find \eqref{model-Lienardm>+} (resp. \eqref{model-Lienardm>-}) in the positive (resp. negative) $x$-direction. It will be more convenient to parameterize the above curves of singularities by $\hat r>0$ instead of $r>0$ where $\hat r$ comes from the transformation in the positive $y$-direction: $x=\frac{\hat x}{\hat r}, y=\frac{1}{\hat r^\frac{m+1}{2}}$. Now, since $\frac{1}{\hat r^\frac{m+1}{2}}=y=\frac{\bar y}{r^{\frac{m+1}{2}}}$ (see the coordinate changes above \eqref{model-Lienardm>+} and \eqref{model-Lienardm>-}), we have the following connection between $r$ and $\hat r$ on the curves of singularities: 
\begin{equation}\label{rhatr1}
r=\hat r^{\frac{m+1}{2(n+1)}}\left(1+\sum_{k=0}^{n}b_kr^{n+1-k}\right)^\frac{1}{n+1}
\end{equation}
\begin{equation}\label{rhatr2}\left(\text{resp. } r=\hat r^{\frac{m+1}{2(n+1)}}\left(1+\sum_{k=0}^{n}b_k(-1)^k r^{n+1-k}\right)^\frac{1}{n+1}\right).
\end{equation}
We denote by $r=\psi_- (\hat r)$ (resp. $r=\psi_+ (\hat r)$), with small $\hat r\ge 0$, the unique solution to \eqref{rhatr1} (resp. \eqref{rhatr2}) with $\psi_- (0)=0$ (resp. $\psi_+ (0)=0$). We use The Implicit Function Theorem. In similar fashion to proving Lemma \ref{label-LemmaPhi}, we can prove
\begin{lemma}\label{label-Lemmapsi} We have $\psi_\pm (\hat r)=\hat r^{\frac{m+1}{2(n+1)}}\left(1+o(1)\right)$ and 
\begin{equation}\label{expansionpsinovo}
    \psi_-(\hat r)-\psi_+(\hat r)=\frac{2}{n+1}\sum_{j=0}^{\frac{n-1}{2}}b_{2j+1}\hat r^{\frac{(n+1-2j)(m+1)}{2(n+1)}}(1+o(1))\nonumber
    \end{equation}
    where $o(1)$-functions tend to $0$ as $\hat r\to 0$.
\end{lemma}
In the rest of this section we will work with $J_\pm(\hat r):=\bar J_\pm(\psi_\pm (\hat r))$ where $\bar J_-$ (resp. $\bar J_+$) is defined in \eqref{SDI999-} (resp. \eqref{SDI999+}).
Assume that a sequence $(\hat r_l)_{l\in\mathbb{N}}$, defined by
\begin{equation}\label{bitno3>-----}
  J_-(\hat r_l)-J_+(\hat r_{l+1})=0 \ \left(\text{i.e., } J_-(\hat r_l)-J_+(\hat r_{l})=J_+(\hat r_{l+1}) -J_+(\hat r_{l})\right), 
\end{equation}
with $\hat r_0>0$ and  $l\in\mathbb N$, monotonically tends to zero as $l\to \infty$. Later it will be clear when this is possible (for each initial point $\hat r_0>0$ small enough).

First we focus on $J_-(\hat r_l)-J_+(\hat r_{l})$. Similarly to \eqref{alig-J} we have 
\begin{align}\label{alig-J>-----}
 &\frac{\left(n+1+\sum_{k=1}^{n}kb_ks^{n+1-k}\right)^2}{1+ \sum_{k=0}^{m-1}a_ks^{m-k}}  \nonumber\\
 & \qquad =\frac{\left(n+1+\sum_{k=1}^{n}kb_k(-1)^ks^{n+1-k}+2\sum_{j=0}^{\frac{n-1}{2}}(2j+1)b_{2j+1}s^{n-2j}\right)^2}{1+ \sum_{k=0}^{m-1}a_k(-1)^{k+1}s^{m-k}+2\sum_{j=0}^{\frac{m-1}{2}}a_{2j}s^{m-2j}}\nonumber\\
 & \qquad = \frac{\left(n+1+\sum_{k=1}^{n}kb_k(-1)^ks^{n+1-k}\right)^2}{1+ \sum_{k=0}^{m-1}a_k(-1)^{k+1}s^{m-k}}-\sum_{j=0}^{\frac{m-1}{2}}a_{2j}s^{m-2j}(2(n+1)^2+O(s))\nonumber\\
      & \qquad \ \ \ +\sum_{j=0}^{\frac{n-1}{2}}b_{2j+1}s^{n-2j}(4(n+1)(2j+1)+O(s)).
\end{align}
Now we get 
\begin{align}\label{alig-J>++++++}
 &J_-(\hat r_l)-J_+(\hat r_{l})=-\int_{0}^{ \psi_- (\hat r_l)}\frac{s^{m-2n-2}\left(n+1+\sum_{k=1}^{n}kb_ks^{n+1-k}\right)^2}{1+ \sum_{k=0}^{m-1}a_ks^{m-k}}ds  \nonumber\\
 & \  \ \ \
 +\int_{0}^{ \psi_+ (\hat r_l)}\frac{s^{m-2n-2}\left(n+1+\sum_{k=1}^{n}kb_k(-1)^ks^{n+1-k}\right)^2}{1+ \sum_{k=0}^{m-1}a_k(-1)^{k+1}s^{m-k}}ds\nonumber\\
 & \  = -\int_{0}^{ \psi_+ (\hat r_l)}s^{m-2n-2}\Bigg(\frac{\left(n+1+\sum_{k=1}^{n}kb_ks^{n+1-k}\right)^2}{1+ \sum_{k=0}^{m-1}a_ks^{m-k}}\nonumber\\
 & \qquad\qquad\qquad\qquad \qquad\qquad - \frac{\left(n+1+\sum_{k=1}^{n}kb_k(-1)^ks^{n+1-k}\right)^2}{1+ \sum_{k=0}^{m-1}a_k(-1)^{k+1}s^{m-k}}\Bigg)ds\nonumber\\
 & \ \ \ \ -\sum_{j=0}^{\frac{n-1}{2}}b_{2j+1}\hat r_l^{\frac{(m-n-2j-1)(m+1)}{2(n+1)}}(2(n+1)+o(1))\nonumber \\
 & \ = \sum_{j=0}^{\frac{m-1}{2}}a_{2j}\frac{2(n+1)^2}{2m-2n-2j-1} \hat r_{l}^{\frac{(2m-2n-2j-1)(m+1)}{2(n+1)}}(1+o(1))\nonumber \\ 
 & \ \ \ \ -\sum_{j=0}^{\frac{n-1}{2}}b_{2j+1}\frac{2(n+1)(m-n+2j+1)}{m-n-2j-1} \hat r_l^{\frac{(m-n-2j-1)(m+1)}{2(n+1)}}(1+o(1))
\end{align}
where $o(1)$-functions tend to $0$ when $\hat r_l\to 0$. In the second step we used Lemma \ref{label-Lemmapsi} and in the last step Lemma \ref{label-Lemmapsi} and \eqref{alig-J>-----}.

Using \eqref{SDI999+}, the term $J_+(\hat r_{l+1}) -J_+(\hat r_{l})$ in \eqref{bitno3>-----} can be written as
\begin{equation}\label{desno+++++>}
J_+(\hat r_{l+1}) -J_+(\hat r_{l})=(n+1)^2\int_{\psi_+(\hat r_{l+1})}^{\psi_+(\hat r_l)}s^{m-2n-2}(1+O(s))ds.
\end{equation}
If we apply the substitution $s=\psi_+(\hat r_l)t$ to \eqref{desno+++++>}, and use \eqref{bitno3>-----} and \eqref{alig-J>++++++}, then we see that $\frac{\psi_+(\hat r_{l+1})}{\psi_+(\hat r_{l})}\to 1$ (i.e. $\frac{\hat r_{l+1}}{\hat r_{l}}\to 1$) as $l\to\infty$ (see also Section \ref{m<frac}). On the other hand, if we use the substitution $s=\psi_+(t)$, then we get 
\begin{align}\label{sub2}
 &\int_{\psi_+(\hat r_{l+1})}^{\psi_+(\hat r_l)}s^{m-2n-2}(1+O(s))ds\nonumber\\
 & \ \ \  =\frac{m+1}{2(n+1)}\int_{\hat r_{l+1}}^{\hat r_l}t^{\left(m-2n-2+\frac{m-2n-1}{m+1}\right)\frac{m+1}{2(n+1)}}(1+o(1))dt
\end{align}
where $o(1)\to 0$ as $t\to 0$. Now, from \eqref{sub2} and the fact that $\frac{\hat r_{l+1}}{\hat r_{l}}\to 1$ as $l\to\infty$ it follows that 
\begin{equation}
\label{subsub2}
\frac{1}{\hat r_l^{\left(m-2n-1\right)\frac{m+1}{2(n+1)}-1}}\int_{\psi_+(\hat r_{l+1})}^{\psi_+(\hat r_l)}s^{m-2n-2}(1+O(s))ds\simeq \hat r_l-\hat r_{l+1}, \ \ l\to\infty.
\end{equation}

In cases (a)--(d) below we suppose that at least one of $j_a$ 
 and $j_b$ is well-defined.
\paragraph{(a) the case $n-2j_b<m-2j_a$ or $a_e=0$.} Because $n-2j_b<m-2j_a$ or $a_e=0$, \eqref{alig-J>++++++} implies that 
\begin{equation}\label{eqcase1>>>>>>>>}
J_-(\hat r)-J_+(\hat r)=-b_{2j_b+1}\frac{2(n+1)(m-n+2j_b+1)}{m-n-2j_b-1} \hat r^{\frac{(m-n-2j_b-1)(m+1)}{2(n+1)}}(1+o(1))
\end{equation}
with $o(1)\to 0$ as $\hat r\to 0$.
 Assume that $b_{2j_b+1}<0$. From \eqref{eqcase1>>>>>>>>} it follows that   $(\hat r_l)_{l\in\mathbb{N}}$ in \eqref{bitno3>-----} is well-defined for each small $\hat r_0>0$, i.e. it tends monotonically to zero as $l\to\infty$. Now, \eqref{bitno3>-----}, \eqref{desno+++++>}, \eqref{subsub2} and \eqref{eqcase1>>>>>>>>} give 
\begin{equation}\label{finally111><}
  \hat r_l-\hat r_{l+1}  \simeq \hat r_l^{(n-2j_b)\frac{m+1}{2(n+1)}+1}, \ l\to \infty.
\end{equation}
Using \cite[Theorem 1]{EZZ} and the fact that $(n-2j_b)\frac{m+1}{2(n+1)}+1>1$ in \eqref{finally111><} (because $j_b\le \frac{n-1}{2}$) we have that $(\hat r_l)_{l\in\mathbb{N}}$ is Minkowski nondegenerate,
\begin{equation}\label{boxdim4444><}
\dim_B(\hat r_l)_{l\in\mathbb{N}}=\frac{(n-2j_b)(m+1)}{(n-2j_b)(m+1)+2(n+1)}
\end{equation}
and these results are independent of the choice of $\hat r_0>0$.

Assume now that $b_{2j_b+1}>0$. Then we use $J_-(\hat r_{l+1})-J_+(\hat r_{l})=0$, instead of \eqref{bitno3>-----}, and we get \eqref{finally111><} and \eqref{boxdim4444><} in a similar way as above.

\paragraph{(b) the case $n-2j_b>m-2j_a$ or $b_o=0$.} Because $n-2j_b>m-2j_a$ or $b_o=0$, \eqref{alig-J>++++++} implies that
\begin{equation}\label{eqcase1>>>>>>>>_____}
J_-(\hat r)-J_+(\hat r)=a_{2j_a}\frac{2(n+1)^2}{2m-2n-2j_a-1} \hat r^{\frac{(2m-2n-2j_a-1)(m+1)}{2(n+1)}}(1+o(1))
\end{equation}
where $o(1)\to 0$ as $\hat r\to 0$.
 Assume that $a_{2j_a}>0$. \eqref{eqcase1>>>>>>>>_____} implies that $(\hat r_l)_{l\in\mathbb{N}}$ in \eqref{bitno3>-----} tends monotonically to zero as $l\to\infty$, for each small $\hat r_0>0$. Now, \eqref{bitno3>-----}, \eqref{desno+++++>}, \eqref{subsub2} and \eqref{eqcase1>>>>>>>>_____} give 
\begin{equation}\label{finally111><666}
  \hat r_l-\hat r_{l+1}  \simeq \hat r_l^{(m-2j_a)\frac{m+1}{2(n+1)}+1}, \ l\to \infty.
\end{equation}
Using \cite[Theorem 1]{EZZ} and $(m-2j_a)\frac{m+1}{2(n+1)}+1>1$ in \eqref{finally111><666} (note that $j_a\le \frac{m-1}{2}$) we have that $(\hat r_l)_{l\in\mathbb{N}}$ is Minkowski nondegenerate,
\begin{equation}\label{boxdim4444><9999}
\dim_B(\hat r_l)_{l\in\mathbb{N}}=\frac{(m-2j_a)(m+1)}{(m-2j_a)(m+1)+2(n+1)}
\end{equation}
and these results are independent of the choice of $\hat r_0>0$.

Assume now that $a_{2j_a}<0$. Then we deal with $J_-(\hat r_{l+1})-J_+(\hat r_{l})=0$, instead of \eqref{bitno3>-----}. Using similar steps as above we find  \eqref{finally111><666} and \eqref{boxdim4444><9999}.

 \paragraph{(c) the case $n-2j_b=m-2j_a$ and $C:=b_{2j_b+1}(m-n+2j_b+1)-a_{2j_a}(n+1)\ne 0$.} $n-2j_b=m-2j_a$ and \eqref{alig-J>++++++} imply
 \begin{equation}
J_-(\hat r)-J_+(\hat r)=-C\frac{2(n+1)}{m-n-2j_b-1} \hat r^{\frac{(m-n-2j_b-1)(m+1)}{2(n+1)}}(1+o(1)).\nonumber
\end{equation}
If $C<0$ (resp. $C>0$), then we have the same analysis and results (\eqref{finally111><} and \eqref{boxdim4444><}) as in case (a) with $b_{2j_b+1}<0$ (resp. $b_{2j_b+1}>0$).

\paragraph{(d) the case $n-2j_b=m-2j_a$ and $C=0$.} This is a topic of further study.

\paragraph{(e) the case $a_e=b_o=0$.} We deal with constant sequences $(\hat r_l)_{l\in\mathbb{N}}$, with trivial Minkowski dimension.

\section{Proof of the main results}
\label{sec-proofs}
In this section we prove Theorems \ref{thm-1}--\ref{thm-3} stated in Section \ref{sec-Motivation}. We use the fractal analysis from Section \ref{sectionFAI} and one important property of the notion of slow divergence integral (see \cite[Chapter 5]{DDR-book-SF}): its invariance under changes of coordinates and time reparameterizations. 

\subsection{Proof of Theorem \ref{thm-1}}\label{sectionproof1}
Let \eqref{model-Lienard2} satisfy \eqref{assum1} and \eqref{assum2}, and $m<2n+1$. We have that $m$ and $n$ are odd and $A=1$. We focus on the fractal analysis of a sequence $U=\{y_l \ | \ l\in\mathbb N\}$, defined by $I_-(y_{l})=I_+(y_{l+1})$ or $I_-(y_{l+1})=I_+(y_{l})$, that tends to $+\infty$, for each initial point $y_0>0$ large enough. Let's recall that $I_-$ (resp. $I_+$) is the slow divergence integral associated to the attracting (resp. repelling) portion of the curve of singularities $y=F(x)$ of \eqref{model-Lienard2}, with $\epsilon=0$, and that $I=I_--I_+$ (see Section \ref{sec-Motivation}).

System \eqref{model-Lienardm<y+}, used in Section \ref{m<frac}, is obtained from \eqref{model-Lienard2}, after the change of coordinates $x=\frac{\bar x}{r}, y=\frac{1}{r^{n+1}}$ and multiplication by $r^n$. We have $F\left(\frac{\Phi_\pm( r)}{ r}\right)=\frac{1}{r^{n+1}}$ where  $\bar x=\Phi_\pm(r)$ is the curve of singularities of \eqref{model-Lienardm<y+} when $\epsilon=0$ (see Section \ref{m<frac}). The slow divergence integral $J_\pm$ associated to $\bar x=\Phi_\pm(r)$ is given in \eqref{bitno2} ($\tilde r>0$ introduced in \eqref{bitno2} is a small constant).

If $U$ is defined by $I_-(y_{l})=I_+(y_{l+1})$ (resp. $I_-(y_{l+1})=I_+(y_{l})$) and if we write $r_l:=\frac{1}{y_l^\frac{1}{n+1}}$, then we get
\begin{align}\label{konacno111}
 I_-&(y_{l})-I_+(y_{l+1})=I_-\left(\frac{1}{r_l^{n+1}}\right)-I_+\left(\frac{1}{r_{l+1}^{n+1}}\right)\nonumber \\ 
 &=-\int_{\frac{\Phi_-(r_l)}{r_l}}^{0}\frac{F'(x)^2}{G(x)}dx+\int_{\frac{\Phi_+(r_{l+1})}{r_{l+1}}}^{0}\frac{F'(x)^2}{G(x)}dx\nonumber \\
 &=-\left(\int_{\frac{\Phi_-(r_l)}{r_l}}^{\frac{\Phi_-(\tilde r)}{\tilde r}}+\int_{\frac{\Phi_-(\tilde r)}{\tilde r}}^{0}\right)\frac{F'(x)^2}{G(x)}dx+\left(\int_{\frac{\Phi_+(r_{l+1})}{r_{l+1}}}^{\frac{\Phi_+(\tilde r)}{\tilde r}}+\int_{\frac{\Phi_+(\tilde r)}{\tilde r}}^{0}\right)\frac{F'(x)^2}{G(x)}dx\nonumber \\
  &=J_-(r_{l})-J_+(r_{l+1})+I\left(\frac{1}{\tilde r^{n+1}}\right)
\end{align}
\begin{equation}\label{konacno222}
    \left(\text{resp. }  I_-(y_{l+1})-I_+(y_{l})=J_-(r_{l+1})-J_+(r_{l})+I\left(\frac{1}{\tilde r^{n+1}}\right) \right).
\end{equation}
In the last step in \eqref{konacno111} we used the above mentioned invariance of the slow divergence integral: $-\int_{\frac{\Phi_\pm(r)}{r}}^{\frac{\Phi_\pm(\tilde r)}{\tilde r}}=J_\pm(r)$, $r<\tilde r$. If we write $\tilde J:=-I\left(\frac{1}{\tilde r^{n+1}}\right)$, then from \eqref{konacno111} (resp.\eqref{konacno222}) it follows that $I_-(y_{l})=I_+(y_{l+1})$ (resp. $I_-(y_{l+1})=I_+(y_{l})$) is equivalent with 
\begin{equation}\label{konacno33333}
    J_-(r_{l})-J_+(r_{l+1})=\tilde J \ \left(\text{resp. } J_-(r_{l+1})-J_+(r_{l})=\tilde J\right).
\end{equation}

Since the Minkowski dimension of $U$ is equal to the Minkowski dimension of $(r_l)_{l\in\mathbb N}$ (see \eqref{def-dim-1}), it suffices to study the Minkowski dimension of $(r_l)_{l\in\mathbb N}$ defined in \eqref{konacno33333}. This has been done in Section \ref{m<frac}.

\begin{remark}\label{remarkkonacno666}
    Similarly to \eqref{konacno111} we can prove that
    \[I(y)=I_-(y)-I_+(y)=J_-(r)-J_+(r)-\tilde J \]
    for all $y=\frac{1}{r^{n+1}}$ large enough.
\end{remark}

\paragraph{Proof of Theorem \ref{thm-1}.1} Suppose that $n-2j_b<m-2j_a$ or $a_e=0$. If $m-n-2j_b-1< 0$, then Theorem \ref{thm-1}.1(a) follows from case (a) in Section \ref{m<frac} and Remark \ref{remarkkonacno666}.  If $m-n-2j_b-1> 0$, then Theorem \ref{thm-1}.1(b) follows from case (b) in Section \ref{m<frac} (note that Remark \ref{remarkkonacno666} implies that $I(y)\to I_*:=\bar J-\tilde J$ as $y\to +\infty$, with $\bar J$ defined in \eqref{alig-JJJ} or \eqref{eqcase1}). 

\paragraph{Proof of Theorem \ref{thm-1}.2} Suppose that $n-2j_b>m-2j_a$ or $b_o=0$. If $2m-2n-2j_a-1<0$, then Theorem \ref{thm-1}.2(a) follows from case (c) in Section \ref{m<frac} and Remark \ref{remarkkonacno666}.  If $2m-2n-2j_a-1>0$, then Theorem \ref{thm-1}.2(b) follows from case (d) in Section \ref{m<frac} and Remark \ref{remarkkonacno666}. 

\paragraph{Proof of Theorem \ref{thm-1}.3} Suppose that $n-2j_b=m-2j_a$ and $C:=b_{2j_b+1}(m-n+2j_b+1)-a_{2j_a}(n+1)\ne 0$. Theorem \ref{thm-1}.3 follows from cases (e) and (f) in Section \ref{m<frac} and Remark \ref{remarkkonacno666}.

\subsection{Proof of Theorem \ref{thm-2}}\label{sectionproof2} 
Let \eqref{model-Lienard2} satisfy \eqref{assum1} and \eqref{assum2}, and $m=2n+1$. We have that $n$ is odd and $A>0$. Let $J_\pm$ be the slow divergence integral given in \eqref{bitno2====}, attached to system \eqref{model-Lienardm===y+}. We obtain \eqref{model-Lienardm===y+} if we apply $x=\frac{\bar x}{r}, y=\frac{1}{r^{n+1}}$ to \eqref{model-Lienard2}, upon multiplication by $r^n$. We get \eqref{konacno33333} and Remark \ref{remarkkonacno666} in the same way as in Section \ref{sectionproof1}.

Remark \ref{remarkbitno} and Remark \ref{remarkkonacno666} imply that $I(y)\to I_*:=\bar J-\tilde J$ as $y\to +\infty$.
Theorem \ref{thm-2}.1 follows from case (a) in Section \ref{section-analysism=}.
 Theorem \ref{thm-2}.2 follows from case (b) in Section \ref{section-analysism=}.
Theorem \ref{thm-2}.3 follows from case (c) in Section \ref{section-analysism=}.
Theorem \ref{thm-2}.4 follows from case (d) in Section \ref{section-analysism=}.

\subsection{Proof of Theorem \ref{thm-3}}\label{sectionproof3}
Let \eqref{model-Lienard2} satisfy \eqref{assum1} and \eqref{assum2}, and $m>2n+1$. We have that $m$ and $n$ are odd and $A=1$. Assume the equation \eqref{balanced-inf} from Section \ref{sec-Motivation} holds, i.e. $I(y)=I_-(y)-I_+(y)$ converges to $0$ as $y\to +\infty$. Like in Sections \ref{sectionproof1} and \ref{sectionproof2} we consider sequences $U=\{y_l \ | \ l\in\mathbb N\}$, defined by $I_-(y_{l})=I_+(y_{l+1})$ or $I_-(y_{l+1})=I_+(y_{l})$, that tend to $+\infty$.

Following Section \ref{section-inf-m>}, if we apply 
$x=\frac{1}{r}, \ y=\frac{\bar y}{r^{\frac{m+1}{2}}}$ (resp. $x=\frac{-1}{r}, \ y=\frac{\bar y}{r^{\frac{m+1}{2}}}$) to \eqref{model-Lienard2}, then we get \eqref{model-Lienardm>+} (resp. system \eqref{model-Lienardm>-}) after multiplication by $r^\frac{m-1}{2}$.
It is clear that $F\left(\frac{1}{\psi_-(\hat r)}\right)=\frac{1}{\hat r^{\frac{m+1}{2}} }$ and $F\left(\frac{-1}{\psi_+(\hat r)}\right)=\frac{1}{\hat r^{\frac{m+1}{2}} }$ where $\psi_-$ and $\psi_+$ are defined after \eqref{rhatr1} and \eqref{rhatr2} and $y=F(x)$ is the curve of singularities of \eqref{model-Lienard2}. $J_-$ and $J_+$ are defined after Lemma \ref{label-Lemmapsi}.

If $U$ is defined by $I_-(y_{l})=I_+(y_{l+1})$ (resp. $I_-(y_{l+1})=I_+(y_{l})$) and if we write $\hat r_l:=\frac{1}{y_l^\frac{2}{m+1}}$, then we get
\begin{align}\label{konacnom>>>111}
 I_-&(y_{l})-I_+(y_{l+1})=I_-\left(\frac{1}{\hat r_l^{\frac{m+1}{2}} }\right)-I_+\left(\frac{1}{\hat r_{l+1}^{\frac{m+1}{2}} }\right)\nonumber \\ 
 &=-\int_{\frac{1}{\psi_-(\hat r_l)}}^{0}\frac{F'(x)^2}{G(x)}dx+\int_{\frac{-1}{\psi_+(\hat r_{l+1})}}^{0}\frac{F'(x)^2}{G(x)}dx\nonumber \\
 &=-\left(\int_{+\infty}^{0}-\int_{+\infty}^{\frac{1}{\psi_-(\hat r_l)}}\right)\frac{F'(x)^2}{G(x)}dx+\left(\int_{-\infty}^{0}-\int_{-\infty}^{\frac{-1}{\psi_+(\hat r_{l+1})}}\right)\frac{F'(x)^2}{G(x)}dx\nonumber \\
  &=-\left(J_-(\hat r_{l})-J_+(\hat r_{l+1})\right) 
\end{align}
\begin{equation}\label{konacno>>>>222}
    \left(\text{resp. }  I_-(y_{l+1})-I_+(y_{l})=-\left(J_-(\hat r_{l+1})-J_+(\hat r_{l})\right)\right).
\end{equation}
In the last step in \eqref{konacnom>>>111} and \eqref{konacno>>>>222} we used \eqref{balanced-inf} and the invariance of the slow divergence integral: $-\int_{+\infty}^{\frac{1}{\psi_-(\hat r)}}=J_-(\hat r)$ and $-\int_{-\infty}^{\frac{-1}{\psi_+(\hat r)}}=J_+(\hat r)$.

From \eqref{def-dim-2}, \eqref{konacnom>>>111} and \eqref{konacno>>>>222} it follows that it suffices to study the Minkowski dimension of $(\hat r_l)_{l\in\mathbb N}$ defined by $J_-(\hat r_{l})-J_+(\hat r_{l+1})=0$ or $J_-(\hat r_{l+1})-J_+(\hat r_{l})=0$. This has been done in Section \ref{section-analysism>}.

Using the above analysis, Theorem \ref{thm-3}.1 follows from case (a) in Section \ref{section-analysism>}, Theorem \ref{thm-3}.2 follows from case (b) in Section \ref{section-analysism>} and Theorem \ref{thm-3}.3 follows from case (c) in Section \ref{section-analysism>}.

\appendix

\section{Family blow-up near infinity for $m>2n+1$}
\label{appendix}
\paragraph{$m$ is odd.} To desingularize \eqref{model-Lienardm>+}, we use the family blow-up \eqref{fam-blow-up} at the origin in $(r,\bar y,\epsilon)$-space. We use different charts. In the family chart $\{\tilde \epsilon=1\}$ we have 
\[(r,\bar y,\epsilon)=(v\tilde r,v^{\frac{m+1}{2}-n-1}\tilde y,v^{m-2n-1})\]
where $(\tilde r,\tilde y)$ is kept in a large compact set. System \eqref{model-Lienardm>+} changes, after division by $v^{\frac{m+1}{2}-n-1}$ and $v\to 0$, into 
\begin{equation}
\label{model-Lienardm>family}
    \begin{vf}
        \dot{\tilde r} &=& -\tilde r\left(\tilde y-\tilde r^{\frac{m+1}{2}-n-1}\right)   \\
        \dot{\tilde y} &=&-A-\frac{m+1}{2}\tilde y \left(\tilde y-\tilde r^{\frac{m+1}{2}-n-1}\right).
    \end{vf}
\end{equation}
When $A=1$, system \eqref{model-Lienardm>family} has no singularities. When $A=-1$, \eqref{model-Lienardm>family} has an attracting node at $(\tilde r,\tilde y)=(0, \sqrt{\frac{2}{m+1}})$  with eigenvalues $(-\sqrt{\frac{2}{m+1}},-\sqrt{2(m+1)})$ and a repelling node at $(\tilde r,\tilde y)=(0, -\sqrt{\frac{2}{m+1}})$  with eigenvalues $(\sqrt{\frac{2}{m+1}},\sqrt{2(m+1)})$.

In the phase directional chart $\{\tilde y=1\}$ we have 
\[(r,\bar y,\epsilon)=(v\tilde r,v^{\frac{m+1}{2}-n-1},v^{m-2n-1}\tilde \epsilon).\]
System \eqref{model-Lienardm>+} changes, after dividing by $v^{\frac{m+1}{2}-n-1}$, into 
\begin{equation}
\label{model-Lienardm>phasey+}
    \begin{vf}
        \dot{\tilde r} &=& \frac{2}{m-2n-1}\tilde r \left(\tilde \epsilon A+O(\tilde r v \tilde \epsilon)\right) +\frac{2(n+1)}{m-2n-1}\tilde r \Psi(\tilde r,v) \\
        \dot{v} &=&-\frac{2}{m-2n-1}v\left(\tilde \epsilon A+O(\tilde r v \tilde \epsilon)+\frac{m+1}{2}\Psi(\tilde r,v)\right)\\
        \dot{\tilde \epsilon} &=&2\tilde \epsilon\left(\tilde \epsilon A+O(\tilde r v \tilde \epsilon)+\frac{m+1}{2}\Psi(\tilde r,v)\right),
    \end{vf}
\end{equation}
where $\Psi(\tilde r,v)=1-\tilde r^{\frac{m+1}{2}-n-1}(1+O(\tilde r v))$. When $v=\tilde\epsilon=0 $, \eqref{model-Lienardm>phasey+} has a hyperbolic saddle at $\tilde r=0$ with eigenvalues $(\frac{2(n+1)}{m-2n-1},-\frac{m+1}{m-2n-1},m+1)$ and a semi-hyperbolic singularity at $\tilde r=1$ with the stable manifold $\{v=\tilde\epsilon=0 \}$ and a two-dimensional center manifold transverse to the stable manifold. Using asymptotic expansions in $\tilde \epsilon$ and the fact that the curve of singularities of \eqref{model-Lienardm>phasey+} is $\{\Psi(\tilde r,v)=0,\tilde\epsilon=0\}$,  the dynamics inside  center manifolds is given by $\{\dot v=v\tilde\epsilon\left(\frac{A}{n+1}+O(v,\tilde\epsilon)\right),\dot{\tilde\epsilon}=-\tilde\epsilon^2\left(\frac{A(m-2n-1)}{n+1}+O(v,\tilde\epsilon)\right)\}$.

In the phase directional chart $\{\tilde y=-1\}$ we have 
\[(r,\bar y,\epsilon)=(v\tilde r,-v^{\frac{m+1}{2}-n-1},v^{m-2n-1}\tilde \epsilon).\]
System \eqref{model-Lienardm>+} changes, after dividing by $v^{\frac{m+1}{2}-n-1}$, into 
\begin{equation}
\label{model-Lienardm>phasey-}
    \begin{vf}
        \dot{\tilde r} &=&- \frac{2}{m-2n-1}\tilde r \left(\tilde \epsilon A+O(\tilde r v \tilde \epsilon)\right) -\frac{2(n+1)}{m-2n-1}\tilde r \Psi(\tilde r,v) \\
        \dot{v} &=&\frac{2}{m-2n-1}v\left(\tilde \epsilon A+O(\tilde r v \tilde \epsilon)+\frac{m+1}{2}\Psi(\tilde r,v)\right)\\
        \dot{\tilde \epsilon} &=&-2\tilde \epsilon\left(\tilde \epsilon A+O(\tilde r v \tilde \epsilon)+\frac{m+1}{2}\Psi(\tilde r,v)\right),
    \end{vf}
\end{equation}
where $\Psi(\tilde r,v)=1+\tilde r^{\frac{m+1}{2}-n-1}(1+O(\tilde r v))$. When $v=\tilde\epsilon=0 $, \eqref{model-Lienardm>phasey-} has a hyperbolic saddle at $\tilde r=0$ with eigenvalues $(-\frac{2(n+1)}{m-2n-1},\frac{m+1}{m-2n-1},-(m+1))$.

We find one extra singularity in the phase directional chart $\{\tilde r=1\}$
\[(r,\bar y,\epsilon)=(v,v^{\frac{m+1}{2}-n-1}\tilde y,v^{m-2n-1}\tilde \epsilon).\]
System \eqref{model-Lienardm>+} changes, after dividing by $v^{\frac{m+1}{2}-n-1}$, into 
\begin{equation}
\label{model-Lienardm>phaser}
    \begin{vf}
        \dot{\tilde y} &=&-\tilde \epsilon A+O( v \tilde \epsilon) -(n+1)\tilde y\left(\tilde y-1+O(v)\right) \\
        \dot{v} &=&-v\left(\tilde y-1+O(v)\right)\\
        \dot{\tilde \epsilon} &=&(m-2n-1)\tilde\epsilon\left(\tilde y-1+O(v)\right).
    \end{vf}
\end{equation}
 When $v=\tilde\epsilon=0 $, \eqref{model-Lienardm>phaser} has a hyperbolic saddle at $\tilde y=0$ with eigenvalues $(n+1,1,2n+1-m)$.
\smallskip

To desingularize \eqref{model-Lienardm>-}, we use the family blow-up \eqref{fam-blow-up} at the origin in $(r,\bar y,\epsilon)$-space. As usual we work with different charts. In the family chart $\{\tilde \epsilon=1\}$ system \eqref{model-Lienardm>-} changes, after division by $v^{\frac{m+1}{2}-n-1}$ and $v\to 0$, into 
\begin{equation}
\label{model-Lienardm>family5}
    \begin{vf}
        \dot{\tilde r} &=& \tilde r\left(\tilde y+(-1)^n\tilde r^{\frac{m+1}{2}-n-1}\right)   \\
        \dot{\tilde y} &=&A+\frac{m+1}{2}\tilde y \left(\tilde y+(-1)^n\tilde r^{\frac{m+1}{2}-n-1}\right).
    \end{vf}
\end{equation}
When $A=1$, system \eqref{model-Lienardm>family5} has no singularities. When $A=-1$, \eqref{model-Lienardm>family5} has a repelling node at $(\tilde r,\tilde y)=(0, \sqrt{\frac{2}{m+1}})$  with the eigenvalues $\sqrt{\frac{2}{m+1}}(1,m+1)$ and an attracting node at $(\tilde r,\tilde y)=(0, -\sqrt{\frac{2}{m+1}})$  with the eigenvalues $-\sqrt{\frac{2}{m+1}}(1,m+1)$.

In the phase directional chart $\{\tilde y=1\}$ system \eqref{model-Lienardm>-} changes, after dividing by $v^{\frac{m+1}{2}-n-1}$, into 
\begin{equation}
\label{model-Lienardm>phasey+5}
    \begin{vf}
        \dot{\tilde r} &=& -\frac{2}{m-2n-1}\tilde r \left(\tilde \epsilon A+O(\tilde r v \tilde \epsilon)\right) -\frac{2(n+1)}{m-2n-1}\tilde r \Psi(\tilde r,v) \\
        \dot{v} &=&\frac{2}{m-2n-1}v\left(\tilde \epsilon A+O(\tilde r v \tilde \epsilon)+\frac{m+1}{2}\Psi(\tilde r,v)\right)\\
        \dot{\tilde \epsilon} &=&-2\tilde \epsilon\left(\tilde \epsilon A+O(\tilde r v \tilde \epsilon)+\frac{m+1}{2}\Psi(\tilde r,v)\right),
    \end{vf}
\end{equation}
where $\Psi(\tilde r,v)=1+\tilde r^{\frac{m+1}{2}-n-1}((-1)^n+O(\tilde r v))$. When $v=\tilde\epsilon=0 $, \eqref{model-Lienardm>phasey+5} has a hyperbolic saddle at $\tilde r=0$ with eigenvalues $(-\frac{2(n+1)}{m-2n-1},\frac{m+1}{m-2n-1},-(m+1))$ and, if $n$ is odd, a semi-hyperbolic singularity at $\tilde r=1$ with the unstable manifold $\{v=\tilde\epsilon=0 \}$ and a two-dimensional center manifold transverse to the unstable manifold. The dynamics inside  center manifolds is given by $\{\dot v=v\tilde\epsilon\left(-\frac{A}{n+1}+O(v,\tilde\epsilon)\right),\dot{\tilde\epsilon}=\tilde\epsilon^2\left(\frac{A(m-2n-1)}{n+1}+O(v,\tilde\epsilon)\right)\}$.

In the phase directional chart $\{\tilde y=-1\}$ system \eqref{model-Lienardm>-} changes, after dividing by $v^{\frac{m+1}{2}-n-1}$, into 
\begin{equation}
\label{model-Lienardm>phasey-5}
    \begin{vf}
        \dot{\tilde r} &=& \frac{2}{m-2n-1}\tilde r \left(\tilde \epsilon A+O(\tilde r v \tilde \epsilon)\right) +\frac{2(n+1)}{m-2n-1}\tilde r \Psi(\tilde r,v) \\
        \dot{v} &=&-\frac{2}{m-2n-1}v\left(\tilde \epsilon A+O(\tilde r v \tilde \epsilon)+\frac{m+1}{2}\Psi(\tilde r,v)\right)\\
        \dot{\tilde \epsilon} &=&2\tilde \epsilon\left(\tilde \epsilon A+O(\tilde r v \tilde \epsilon)+\frac{m+1}{2}\Psi(\tilde r,v)\right),
    \end{vf}
\end{equation}
where $\Psi(\tilde r,v)=1-\tilde r^{\frac{m+1}{2}-n-1}((-1)^n+O(\tilde r v))$. When $v=\tilde\epsilon=0 $, \eqref{model-Lienardm>phasey-5} has a hyperbolic saddle at $\tilde r=0$ with eigenvalues $(\frac{2(n+1)}{m-2n-1},-\frac{m+1}{m-2n-1},m+1)$ and, if $n$ is even, a semi-hyperbolic singularity at $\tilde r=1$ with the stable manifold $\{v=\tilde\epsilon=0 \}$ and a two-dimensional center manifold transverse to the stable manifold. The dynamics inside  center manifolds is given by $\{\dot v=v\tilde\epsilon\left(\frac{A}{n+1}+O(v,\tilde\epsilon)\right),\dot{\tilde\epsilon}=\tilde\epsilon^2\left(-\frac{A(m-2n-1)}{n+1}+O(v,\tilde\epsilon)\right)\}$.

We find one extra singularity in the phase directional chart $\{\tilde r=1\}$.
System \eqref{model-Lienardm>-} changes, after dividing by $v^{\frac{m+1}{2}-n-1}$, into 
\begin{equation}
\label{model-Lienardm>phaser5}
    \begin{vf}
        \dot{\tilde y} &=&\tilde \epsilon A+O( v \tilde \epsilon) +(n+1)\tilde y\left(\tilde y+(-1)^n+O(v)\right) \\
        \dot{v} &=&v\left(\tilde y+(-1)^n+O(v)\right)\\
        \dot{\tilde \epsilon} &=&-(m-2n-1)\tilde\epsilon\left(\tilde y+(-1)^n+O(v)\right).
    \end{vf}
\end{equation}
 When $v=\tilde\epsilon=0 $, \eqref{model-Lienardm>phaser5} has a hyperbolic saddle at $\tilde y=0$ with eigenvalues $(-1)^n(n+1,1,2n+1-m)$.

 \smallskip

 \paragraph{$m$ is even.} To desingularize \eqref{model-Lienardmeven>+}, we use the family blow-up \eqref{fam-blow-up--} at the origin in $(r,\bar y,\epsilon)$-space. In the family chart $\{\tilde \epsilon=1\}$ system \eqref{model-Lienardmeven>+} changes, after division by $v^{m-2n-1}$ and $v\to 0$, into 
\begin{equation}
\label{model-Lienardmeven>family}
    \begin{vf}
        \dot{\tilde r} &=& -\frac{1}{2}\tilde r\left(\tilde y-\tilde r^{m-2n-1}\right)   \\
        \dot{\tilde y} &=&-A-\frac{m+1}{2}\tilde y \left(\tilde y-\tilde r^{m-2n-1}\right).
    \end{vf}
\end{equation}
System \eqref{model-Lienardmeven>family} has no singularities because $A=1$.

In the phase directional chart $\{\tilde y=1\}$ \eqref{model-Lienardmeven>+} changes, after dividing by $v^{m-2n-1}$, into 
\begin{equation}
\label{model-Lienardm>phaseeveny+}
    \begin{vf}
        \dot{\tilde r} &=& \frac{1}{m-2n-1}\tilde r \left(\tilde \epsilon A+O(\tilde r v \tilde \epsilon)\right) +\frac{n+1}{m-2n-1}\tilde r \Psi(\tilde r,v) \\
        \dot{v} &=&-\frac{1}{m-2n-1}v\left(\tilde \epsilon A+O(\tilde r v \tilde \epsilon)+\frac{m+1}{2}\Psi(\tilde r,v)\right)\\
        \dot{\tilde \epsilon} &=&2\tilde \epsilon\left(\tilde \epsilon A+O(\tilde r v \tilde \epsilon)+\frac{m+1}{2}\Psi(\tilde r,v)\right),
    \end{vf}
\end{equation}
where $\Psi(\tilde r,v)=1-\tilde r^{m-2n-1}(1+O(\tilde r v))$. When $v=\tilde\epsilon=0 $, \eqref{model-Lienardm>phaseeveny+} has a hyperbolic saddle at $\tilde r=0$ with eigenvalues $(\frac{n+1}{m-2n-1},-\frac{m+1}{2(m-2n-1)},m+1)$ and a semi-hyperbolic singularity at $\tilde r=1$ with the stable manifold $\{v=\tilde\epsilon=0 \}$ and a two-dimensional center manifold transverse to the stable manifold. The dynamics inside center manifolds is given by $\{\dot v=v\tilde\epsilon\left(\frac{A}{2(n+1)}+O(v,\tilde\epsilon)\right),\dot{\tilde\epsilon}=-\tilde\epsilon^2\left(\frac{A(m-2n-1)}{n+1}+O(v,\tilde\epsilon)\right)\}$.

Since $m-2n-1$ is odd, we can cover the phase directional chart $\{\tilde y=-1\}$ by applying $(t,\tilde r,v)\mapsto (-t,-\tilde r,-v)$ to \eqref{model-Lienardm>phaseeveny+}. When $v=\tilde\epsilon=0 $, we find
 a hyperbolic saddle at $\tilde r=0$ with eigenvalues $(-\frac{n+1}{m-2n-1},\frac{m+1}{2(m-2n-1)},-(m+1))$.

We find one extra singularity in the phase directional chart $\{\tilde r=1\}$
in which system \eqref{model-Lienardmeven>+} changes, after dividing by $v^{m-2n-1}$, into 
\begin{equation}
\label{model-Lienardmeven>phaser}
    \begin{vf}
        \dot{\tilde y} &=&-\tilde \epsilon A+O( v \tilde \epsilon) -(n+1)\tilde y\left(\tilde y-1+O(v)\right) \\
        \dot{v} &=&-\frac{1}{2}v\left(\tilde y-1+O(v)\right)\\
        \dot{\tilde \epsilon} &=&(m-2n-1)\tilde\epsilon\left(\tilde y-1+O(v)\right).
    \end{vf}
\end{equation}
 When $v=\tilde\epsilon=0 $, \eqref{model-Lienardmeven>phaser} has a hyperbolic saddle at $\tilde y=0$ with eigenvalues $(n+1,\frac{1}{2},2n+1-m)$.
\smallskip

To desingularize \eqref{model-Lienardmeven>-}, we use the family blow-up \eqref{fam-blow-up--} at the origin in $(r,\bar y,\epsilon)$-space. In the family chart $\{\tilde \epsilon=1\}$ system \eqref{model-Lienardmeven>-} changes, after division by $v^{m-2n-1}$ and $v\to 0$, into 
\begin{equation}
\label{model-Lienardm>family10}
    \begin{vf}
        \dot{\tilde r} &=&\frac{1}{2} \tilde r\left(\tilde y+(-1)^n\tilde r^{m-2n-1}\right)   \\
        \dot{\tilde y} &=&-A+\frac{m+1}{2}\tilde y \left(\tilde y+(-1)^n\tilde r^{m-2n-1}\right).
    \end{vf}
\end{equation}
Since $A=1$, system \eqref{model-Lienardm>family10} has a repelling node at $(\tilde r,\tilde y)=(0, \sqrt{\frac{2}{m+1}})$  with the eigenvalues $\sqrt{\frac{2}{m+1}}(\frac{1}{2},m+1)$ and an attracting node at $(\tilde r,\tilde y)=(0, -\sqrt{\frac{2}{m+1}})$  with the eigenvalues $-\sqrt{\frac{2}{m+1}}(\frac{1}{2},m+1)$.

In the phase directional chart $\{\tilde y=1\}$ system \eqref{model-Lienardmeven>-} changes, after dividing by $v^{m-2n-1}$, into 
\begin{equation}
\label{model-Lienardm>phasey+10}
    \begin{vf}
        \dot{\tilde r} &=& \frac{1}{m-2n-1}\tilde r \left(\tilde \epsilon A+O(\tilde r v \tilde \epsilon)\right) -\frac{n+1}{m-2n-1}\tilde r \Psi(\tilde r,v) \\
        \dot{v} &=&\frac{1}{m-2n-1}v\left(-\tilde \epsilon A+O(\tilde r v \tilde \epsilon)+\frac{m+1}{2}\Psi(\tilde r,v)\right)\\
        \dot{\tilde \epsilon} &=&-2\tilde \epsilon\left(-\tilde \epsilon A+O(\tilde r v \tilde \epsilon)+\frac{m+1}{2}\Psi(\tilde r,v)\right),
    \end{vf}
\end{equation}
where $\Psi(\tilde r,v)=1+\tilde r^{m-2n-1}((-1)^n+O(\tilde r v))$. When $v=\tilde\epsilon=0 $, \eqref{model-Lienardm>phasey+10} has a hyperbolic saddle at $\tilde r=0$ with eigenvalues $(-\frac{n+1}{m-2n-1},\frac{m+1}{2(m-2n-1)},-(m+1))$ and, if $n$ is odd, a semi-hyperbolic singularity at $\tilde r=1$ with the unstable manifold $\{v=\tilde\epsilon=0 \}$ and a two-dimensional center manifold transverse to the unstable manifold. The dynamics inside  center manifolds is given by $\{\dot v=v\tilde\epsilon\left(\frac{A}{2(n+1)}+O(v,\tilde\epsilon)\right),\dot{\tilde\epsilon}=-\tilde\epsilon^2\left(\frac{A(m-2n-1)}{n+1}+O(v,\tilde\epsilon)\right)\}$.

We cover the phase directional chart $\{\tilde y=-1\}$ by applying $(t,\tilde r,v)\mapsto (-t,-\tilde r,-v)$ to \eqref{model-Lienardm>phasey+10}. When $v=\tilde\epsilon=0 $, we find a hyperbolic saddle at $\tilde r=0$ with eigenvalues $(\frac{n+1}{m-2n-1},-\frac{m+1}{2(m-2n-1)},m+1)$ and, if $n$ is even, a semi-hyperbolic singularity at $\tilde r=1$ with the stable manifold $\{v=\tilde\epsilon=0 \}$ and a two-dimensional center manifold transverse to the stable manifold. The dynamics inside  center manifolds is given by $\{\dot v=-v\tilde\epsilon\left(\frac{A}{2(n+1)}+O(v,\tilde\epsilon)\right),\dot{\tilde\epsilon}=\tilde\epsilon^2\left(\frac{A(m-2n-1)}{n+1}+O(v,\tilde\epsilon)\right)\}$.

We find one extra singularity in the phase directional chart $\{\tilde r=1\}$.
System \eqref{model-Lienardmeven>-} changes, after dividing by $v^{m-2n-1}$, into 
\begin{equation}
\label{model-Lienardm>phaser100}
    \begin{vf}
        \dot{\tilde y} &=&-\tilde \epsilon A+O( v \tilde \epsilon) +(n+1)\tilde y\left(\tilde y+(-1)^n+O(v)\right) \\
        \dot{v} &=&\frac{1}{2}v\left(\tilde y+(-1)^n+O(v)\right)\\
        \dot{\tilde \epsilon} &=&-(m-2n-1)\tilde\epsilon\left(\tilde y+(-1)^n+O(v)\right).
    \end{vf}
\end{equation}
 When $v=\tilde\epsilon=0 $, \eqref{model-Lienardm>phaser100} has a hyperbolic saddle at $\tilde y=0$ with eigenvalues $(-1)^n(n+1,\frac{1}{2},2n+1-m)$.

\section*{Declarations}
 
\textbf{Ethical Approval} \ 
Not applicable.
 \\
\\
 \textbf{Competing interests} \  
The authors declare that they have no conflict of interest.\\
 \\
\textbf{Authors' contributions} \  All authors conceived of the presented idea, developed the theory, performed the computations and
contributed to the final manuscript.  \\ 
\\ 
\textbf{Funding} \
The research of R. Huzak and G. Radunovi\'{c}  was supported by: Croatian Science Foundation (HRZZ) grant
PZS-2019-02-3055 from “Research Cooperability” program funded by the European Social Fund. Additionally, the research of G. Radunovi\'{c} was partially
supported by the HRZZ grant UIP-2017-05-1020.\\
 \\
\textbf{Availability of data and materials}  \
Not applicable.

\bibliographystyle{plain}
\bibliography{bibtex}
\end{document}